\documentclass{amsart}
\usepackage[margin=3cm]{geometry}
\usepackage{amssymb,amsmath,amsfonts,mathtools,latexsym}
\usepackage[pdftex]{hyperref} 
\usepackage{tikz-cd} 
\usepackage{enumerate} 
\newtheorem{theorem}{Theorem}[section]
\newtheorem{corollary}[theorem]{Corollary}
\newtheorem{definition}[theorem]{Definition}
\newtheorem{example}[theorem]{Example}
\newtheorem{lemma}[theorem]{Lemma}
\newtheorem{problem}{Problem}
\newtheorem{proposition}[theorem]{Proposition}
\newtheorem{remark}[theorem]{Remark}

\begin{document}

\title[The Graded Algebras with a Graded Identity of Degree 2]{The Graded Algebras with a Graded Identity of Degree 2}
\date{June 13, 2025}

\author[De Fran\c{c}a]{Antonio de Fran\c{c}a$^\dag$}
\address{Department of Mathematics, Federal Univ. of Campina Grande, 58429-970 Campina Grande -- PB, Brazil}
\email{\href{mailto: a.defranca@yandex.com}{a.defranca@yandex.com}} 
\urladdr{\href{https://sites.google.com/view/adefranca}{https://sites.google.com/view/adefranca}}
\thanks{$^\dag$The author was partially supported by Para\'iba State Research Foundation (FAPESQ), Grant \#2023/2158.}

\begin{abstract}
This paper is devoted to the study of graded associative algebras that satisfy a graded polynomial identity of degree $2$. 
Let $\mathsf{G}$ be a finite abelian group, $\mathbb{F}$ a field of characteristic zero and $\mathfrak{A}$ a $\mathsf{G}$-graded $\mathbb{F}$-algebra. 
We prove that, for $\mathbb{F}$ algebraically closed, if $\mathfrak{A}_e$ satisfies a polynomial identity $g=g(x_1^{(e)}, \dots, x_n^{(e)})\in\mathbb{F}\langle X^\mathsf{G} \rangle$ of degree $2$, then $\mathfrak{A}$ is either nilpotent or has commutative neutral component, 
and we ensure that the $\mathsf{G}$-graded variety $\mathfrak{W}^\mathsf{G}$ determined by $g$ is equal to either $\mathsf{var}^\mathsf{G}([x^{(e)},y^{(e)}])$ or $\mathsf{var}^\mathsf{G}(N)$ for some nilpotent $\mathsf{G}$-graded algebra $N$. 
Posteriorly, we investigate the implications of $\mathfrak{A}_e$ being central in $\mathfrak{A}$. The results obtained allow us to prove that, when $\mathsf{G}$ is finite cyclic, if $\mathfrak{A}$ is finitely generated and $\mathfrak{A}_e$ is central in $\mathfrak{A}$, then the commutator ideal of $\mathfrak{A}$ is nilpotent, and the algebra $\mathfrak{A}^{(-)}=(\mathfrak{A},[\ ,\ ])$ is a solvable Lie algebra, 
and, if $\mathsf{G}$ has odd order, then $[x_1,x_2][x_3,x_4]\cdots[x_{2d-1},x_{2d}]\equiv0$ in $\mathfrak{A}$, for some $d\in\mathbb{N}$.
\end{abstract}


\maketitle


\section{Introduction}

Algebraic structures with gradings by groups ensure a very rich field of research, in the theory ring, because from a structure of grading we can deduce properties of the ordinary object (i.e. object without grading). From this, one of central problems in the study of graded algebras is to obtain non-graded (ordinary) properties from the analysis of gradings of a  given algebra and vice versa. 
In \cite{BergCohe86}, J. Bergen and M. Cohen (1986) showed that, given a $\mathsf{G}$-graded algebra $\mathfrak{A}=\bigoplus_{\xi\in \mathsf{G}} \mathfrak{A}_\xi$, where $\mathsf{G}$ is a finite group with neutral element $e$, if $\mathfrak{A}_e$ is  a $PI$-algebra, then $\mathfrak{A}$ is also a $PI$-algebra. Already in \cite{BahtGiamRile98}, Yu. Bahturin, A. Giambruno and D. Riley (1998) deduced, in addition to this last result, a bound for the minimal degree of the polynomial identity satisfied by $\mathfrak{A}$.

In \cite{AmitLevi50}, S. Amitsur and J. Levitzki (1950) presented \textit{minimal identities} for matrix algebra of degree $n$, where a minimal identity of an algebra $\mathfrak{A}$ is a (nonzero) identity polynomial of degree minimal which $\mathfrak{A}$ satisfies. Obviously, given a $PI$-algebra $\mathfrak{A}\neq\{0\}$, a minimal identity for $\mathfrak{A}$ has degree at least $2$. 
In \cite{ShesZhuk09}, I. Shestakov and N. Zhukavets (2009) proved an analogue for octonion algebras of the famous Amitsur-Levitsky skew-symmetric identity: $\mathbb{O}$ satisfies the (minimal) identity polynomial $\sum_{\sigma} (-1)^{\sigma} (x_{\sigma(1)}x_{\sigma(2)})(x_{\sigma(3)}, x_{\sigma(4)}, x_{\sigma(5)})\equiv0$, which is a skew-symmetric identity of degree $5$, where $(a,b,c)=(a,b)c-a(b,c)$ is the associator. Recall that an \textit{octonion algebra} $\mathbb{O}$ is a composition algebra (i.e. $\mathbb{O}$ has a nondegenerate quadratic form $N$ satisfying $N(ab)=N(a)N(b)$ for any $a,b\in \mathbb{O}$) over a field $\mathbb{F}$ that has dimension $8$ over $\mathbb{F}$. See also the work \cite{ShesZhuk04}, 2004, due to I. Shestakov and N. Zhukavets.  

In 2014, I. Shestakov proposed to I. Sviridova to study the following problem: ``\textit{What is the minimal degree of non-graded identity that $\mathfrak{A}$ satisfies? Which identities does $\mathfrak{A}$ satisfies?}''. In this context, I. Sviridova and O. Finogenova studied this problem and proved, among other results, that ``\textit{if $\mathfrak{A}$ is a $\mathbb{Z}_2$-graded associative algebra such that $\mathfrak{A}_0$ satisfies a homogeneous identity of the $2$nd degree, then $\mathfrak{A}$ satisfies a non-graded identity of the degree $4$ or $5$}'' (results not yet published). Other result proved by I. Sviridova and O. Finogenova is the following: ``\textit{if $\mathsf{char}(\mathbb{F})\neq2$, then $\mathsf{var}^{\mathbb{Z}_2}([x^{(0)},y^{(0)}])=\mathsf{var}^{\mathbb{Z}_2}(M_{1,1}(\mathbb{F})\oplus M_{1,1}(\mathsf{E})$}'', where $\mathsf{E}$ is the Grassmann algebra.  
In \cite{Mardua01}, A. de Fran\c{c}a and I. Sviridova (2022) proved that a ring $\mathfrak{R}$ with a finite $\mathsf{G}$-grading of order $d$ is nilpotent with $\mathsf{nd}(\mathfrak{A})\leq3d$ when its neutral component $\mathfrak{R}_e$ is nil of index $2$ and has characteristic different from $2$.

In this work, our interest is to study the following question:

	\vspace{2pt}
	\noindent{\bf Problem \ref{3.01}.} 
\textit{What can we say about an associative algebra $\mathfrak{A}$ graded by a group $\mathsf{G}$ when its neutral component $\mathfrak{A}_e$ satisfies a polynomial identity $g$ of degree $2$? Which ordinary identities does $\mathfrak{A}$ satisfy?}
	\vspace{-6pt}

Basically, we have studied and answered what are the consequences on a graded algebra $\mathfrak{A}$ when it satisfies some graded polynomial identities of degree $2$. We have also studied the graded variety of these algebras.



Let us now introduce another problem that is also the objective of this work. This problem relates rings admitting certain automorphisms, commutator ideal of a ring and graded rings with central neutral component. 
In \cite{Jaco55}, N. Jacobson (1955) proved that if $\mathfrak{L}$ is a Lie algebra with an automorphism $\sigma$ of prime period $l$ and $\sigma$ and has no fixed points $\neq0$, then $\mathfrak{L}$ is nilpotent. 
Already in \cite{Higm57}, G. Higman (1957) showed that if an associative ring has an automorphism of prime order $p$ which leaves fixed no element except zero, it is nilpotent of class at most $p-1$. 
The analogue of these results for finite groups was proved by J. Thompson in 1959. In \cite{Thom59}, he proved that if $G$ is a finite group with a fixed-point-free automorphism of prime order, then $G$ is nilpotent. 
%
%
%
Afterwards in \cite{Krek67}, V.A. Kreknin showed that any finite-dimensional Lie algebra $\mathfrak{L}$ over an arbitrary field admitting a regular automorphism $\varphi$ is solvable. See also the works \cite{Krek63}, due to V.A. Kreknin (1963), and \cite{KrekKost63}, due to V.A. Kreknin and A.I. Kostrikin (1963), and \cite{Khuk92} due to E.I. Khukhro (1992). 
%

On the other hand, answering whether the commutator ideal is nil, N. Herstein (1962) proved in \cite{Hers62} that if $\mathfrak{R}$ is a ring which admits an automorphism of period $2$ or $3$ all of whose fixed-points are in the center of $\mathfrak{R}$, then the commutator ideal of $\mathfrak{R}$ is a nil ideal. He conjectured that this result is hold in the general case of a ring admitting an automorphism of prime period $p$ all of whose fixed-points lie in the center of the ring.  
So, later, G.M. Bergman and I.M. Isaacs (1973) proved, in \cite{BergIsaa73}, that if $\mathfrak{R}$ is a (free $n$-torsion, unitary) ring graded by a finite cyclic group $\mathsf{G}$ of order $n$ such that $\mathfrak{R}_e$ is central, then the commutator ideal of  $\mathfrak{R}$ is nil. 
Already in \cite{Maka05}, N.Yu. Makarenko (2005) proved that, given a $(\mathbb{Z}/p\mathbb{Z})$-graded Lie ring (algebra) $L = L_0 \oplus L_1 \oplus \cdots \oplus L_{p-1}$ such that $[L_s,L_t]\subseteq L_{s+t(mod\ p)}$, if the component $L_0$ is finite of order $m$, (i.e. a vector space of finite dimension $m$), then $L$ has a nilpotent ideal of nilpotency class bounded by a function of $p$, whose index in the additive group $L$ (codimension) is bounded by a function of $m$ and $p$. 

In the context of Novikov algebras, in \cite{UmirZhel21}, U. Umirbaev and V. Zhelyabin (2021) proved that if $N$ is a $G$-graded Novikov $K$-algebra with solvable $0$-component $N_0$ and the characteristic of the field $K$ does not divide the order of $G$, where $G$ is a finite additive abelian group, then $N$ is solvable. 
Posteriorly, in \cite{TuleUmirZhel23}, K. Tulenbaev, U. Umirbaev and V. Zhelyabin (2023) showed that if $N$ is a Lie-solvable Novikov algebra over a field of characteristic $\neq2$, then the ideal $[N,N]$ is right nilpotent. 
Recall that a nonassociative algebra $N$ over a field $\mathbb{K}$ is called \textit{a Novikov algebra} if it satisfies the identities $(x,y,z)=(y,x,z)$ (left symmetry) and $(xy)z=(xz)y$ (right commutativity), where $(x,y,z)=(xy)z-x(yz)$ is the associator of elements $x,y,z\in N$.

In this way, we have the second central problem of this work:

	\vspace{4pt}
	\noindent{\bf Problem \ref{3.29}.} 
\textit{If $\mathfrak{A}$ is a $\mathsf{G}$-graded algebra such that its neutral component $\mathfrak{A}_e$ is central, then is $\mathfrak{A}^{(-)}$ a solvable/nilpotent Lie algebra? And about the commutator ideal of $\mathfrak{A}$, is it a nilpotent algebra?} 
	\vspace{4pt}

This paper is devoted to the study of \textbf{Problems \ref{3.01} and \ref{3.29}}, and is organized as follows. In \S \ref{preliminar}, we  recall some definitions, define the $\mathsf{f_G}$-commutator of a graded  algebra, give some examples, and prove the first results. Already in \S \ref{main}, the text is dedicated to the main results of the work. These results attempt to answer the Problems \ref{3.01} and \ref{3.29}. One of the results that we have proven is the following:

	\vspace{4pt}
	\noindent{\bf Corollary \ref{3.32}.} 
\textit{Let $\mathsf{G}$ be a finite abelian group, $\mathbb{F}$ an algebraically closed field of characteristic zero, and $\mathfrak{A}$ any associative $\mathsf{G}$-graded algebra. If $\mathfrak{A}_e$ satisfies a polynomial identity of degree $2$, then either $\mathfrak{A}$ is a nilpotent algebra or $\mathfrak{A}_e$ is a commutative algebra.}
	\vspace{4pt}

We notice that the last result is a consequence of Theorem \ref{3.21}, in \S\ref{subsecvariety}, which ensures (under the same conditions of Corollary \ref{3.32}) that the $\mathsf{G}$-graded variety $\mathfrak{W}^\mathsf{G}$ determined by a graded polynomial $g=g(x_1^{(e)}, \dots, x_n^{(e)})\in\mathbb{F}\langle X^\mathsf{G} \rangle$ of degree $2$ is equal to either $\mathsf{var}^\mathsf{G}([x^{(e)},y^{(e)}])$ or $\mathsf{var}^\mathsf{G}(N)$ for some nilpotent $\mathsf{G}$-graded algebra $N$.

We conclude this work with \S\ref{subsecringcentral}, where our main results are stated and proved. The subsection \S\ref{subsecringcentral} begins with Theorem \ref{3.06}, which states that if  $\mathfrak{R}$ is an associative ring graded by a finite cancellative monoid $\mathsf{S}$ of order $d\in\{1,2,3\}$ such that $\mathfrak{R}_e$ is central in $\mathfrak{R}$, then $[x_1, \dots, x_{d+1}]\equiv0$ in $\mathfrak{R}$. Already Proposition \ref{3.28} establishes that this result does not work when $d\geq4$. Below, we show the main result of this subsection, which also answers Problem \ref{3.29}. 

	\vspace{4pt}
	\noindent{\bf Theorem \ref{3.12}.} 
\textit{Let $\mathbb{F}$ be a field of characteristic zero, $\mathsf{G}$ a finite cyclic group, $\mathfrak{A}$ a finitely generated $\mathbb{F}$-algebra with a $\mathsf{G}$-grading $\Gamma$. Suppose that $\mathfrak{A}_e$ is central in $\mathfrak{A}$. The commutator ideal of $\mathfrak{A}$ is nilpotent. Moreover, $\mathfrak{A}^{(-)}$ is a solvable Lie algebra. In addition, if the support of $\Gamma$ has at most 3 elements, then $\mathfrak{A}^{(-)}$ is a nilpotent Lie algebra.}
	\vspace{4pt}

Finally, in \cite{Mardua02}, A. de Fran\c{c}a and I. Sviridova proved that, for $\mathbb{F}$ an algebraically closed field with $\mathsf{char}(\mathbb{F})=0$ and $\mathsf{G}=\mathbb{Z}_p$ a finite group with $\mathsf{gcd}(p,2)=1$, if $\mathfrak{A}$ is a $\mathsf{G}$-graded algebra such that $\mathfrak{A}_e$ is central in $\mathfrak{A}$, then $\mathfrak{A}$ satisfies the polynomial identity $[x_1,x_2,x_3][x_4,x_5,x_6]\cdots[x_{n-2},x_{n-1},x_n]\in \mathbb{F}\langle X\rangle$ for some $n\in\mathbb{N}$. The results obtained along this work allow us to improve this last result. In Theorem \ref{3.30}, we have proved that if $\mathbb{F}$ is any field of characteristic zero, $\mathsf{G}$ is a finite cyclic group of odd order, and $\mathfrak{A}$ is an algebra with a $\mathsf{G}$-grading such that $\mathfrak{A}_e$ is central in $\mathfrak{A}$, then the commutator ideal of $\mathfrak{A}$ is nilpotent and $[x_1,x_2][x_3,x_4]\cdots[x_{2d-1},x_{2d}]\equiv0$ in $\mathfrak{A}$ for some $d\in\mathbb{N}$. 

%
%

%
%
%
\section{Preliminaries}\label{preliminar}

Let $\mathbb{F}$ be a field, $\mathsf{G}$ a group and $\mathfrak{A}$ an associative $\mathbb{F}$-algebra with a $\mathsf{G}$-grading. Recall that a $\mathsf{G}$-grading on $\mathfrak{A}$ is a decomposition $\Gamma: \mathfrak{A} = \bigoplus_{\xi\in \mathsf{G}} \mathfrak{A}_\xi$ that satisfies $\mathfrak{A}_\xi \mathfrak{A}_\zeta \subseteq \mathfrak{A}_{\xi\zeta}$, for any $\xi, \zeta \in \mathsf{G}$, where $\mathfrak{A}_\xi$'s are vector subspaces of $\mathfrak{A}$. 
Now, let $\mathbb{F}\langle X \rangle$ be the free associative algebra, generated freely by the set $X=\{x_1,x_2,\dots\}$, a countable set of indeterminants, and $\mathbb{F}\langle X^{\mathsf{G}} \rangle$ the free graded associative algebra, generated freely by the set $X^{\mathsf{G}}$, where $X^{\mathsf{G}}=\bigcup_{\xi\in\mathsf{G}} X_\xi$, $X_\xi=\{x_{1 \xi}, x_{2\xi}, \dots\}$, a countable set of graded indeterminants. An element of $\mathbb{F}\langle X \rangle$ (resp. $\mathbb{F}\langle X^{\mathsf{G}} \rangle$) is called \textit{a polynomial} (resp. \textit{a graded polynomial}). When no confusion can arise, a polynomial $g=g(x_{i_1\xi_1},\dots,x_{i_n\xi_n})$ of $\mathbb{F}\langle X^{\mathsf{G}}\rangle$ will be written as $g=g(x_1^{(\xi_1)},\dots,x_n^{(\xi_n)})$. 
We say that $g=g(x_1,\dots,x_n)\in\mathbb{F}\langle X \rangle$ is a \textit{polynomial identity} for $\mathfrak{A}$, denoted by $g\equiv0$ in $\mathfrak{A}$, if $g(a_1,\dots,a_n)=0$ for any $a_1,\dots,a_n\in\mathfrak{A}$. Analogously, we say that $w=w(x_{i_1\xi_1},\dots,x_{i_n\xi_n})\in\mathbb{F}\langle X^{\mathsf{G}} \rangle$ is a \textit{graded polynomial identity} for $\mathfrak{A}$, denoted by $w\equiv_{\mathsf{G}}0$ in $\mathfrak{A}$, if $w(b_{\xi_1},\dots,b_{\xi_n})=0$ for any $b_{\xi_1}\in\mathfrak{A}_{\xi_1}$, $\dots$, $b_{\xi_n}\in\mathfrak{A}_{\xi_n}$. 
For further reading, as well as an overview, on graded algebras, see \cite{ElduKoch13} and \cite{NastOyst11}; on the free $\mathsf{G}$-graded associative $\mathbb{F}$-algebra $\mathbb{F}\langle X^{\mathsf{G}} \rangle$, see \cite{GiamZaic05}, p.66, and \cite{NastOyst04}, Proposition 2.3.1, p.22; and on (graded) polynomial identities, see \cite{Dren00}, \cite{GiamZaic05}, \cite{Jaco75} and \cite{Rowe91}.

\begin{definition}\label{3.26}
Considering $\mathbb{F}\langle X^{\mathsf{G}} \rangle$ with its $\mathsf{G}$-grading, a graded polynomial $g\in\mathbb{F}\langle X^{\mathsf{G}} \rangle$ is called $\mathsf{G}$-homogeneous of degree $\xi$ if it is a summand of graded monomials of degree $\xi\in\mathsf{G}$. When no confusion can arise, we say ``$g$ is homogeneous of degree $\xi$'' or still ``$g$ is a homogeneous graded polynomial''.
\end{definition}

In \cite{Mardua01}, A. de Fran\c{c}a and I. Sviridova introduced the definition of an $\mathsf{f}$-commutator of a ring $\mathfrak{R}$ (see Definitions 2.4 and 2.5 in \cite{Mardua01}). Here, let us generalize this concept. Let $\mathsf{G}$ be a group and $\mathfrak{A}$ an algebra with a $\mathsf{G}$-grading $\Gamma$. An \textit{$\mathsf{f_G}$-commutator in the $\mathsf{G}$-grading of $\mathfrak{A}$ which depends of $\Gamma$}, denoted by $[\ ,\ ]_{\mathsf{f_G}}$, is a map from $ \bigcup_{\xi,\zeta\in\mathsf{G}}\left(\mathfrak{A}_\xi \times \mathfrak{A}_\zeta \right)$ into $\mathfrak{A}$ defined by 
$[a_\xi,b_\zeta]_{\mathsf{f_G}}=\mathsf{f_G}(\xi,\zeta)a_\xi b_\zeta-\mathsf{f_G}(\zeta,\xi)b_\zeta a_\xi $,
where $\mathsf{f_G}$ is a map defined from $\mathsf{G}\times\mathsf{G}$ in some semigroup $\mathfrak{S}$ which acts on the left of $\mathfrak{A}$. It is immediate that $\left[a_\xi,b_\zeta \right]_{\mathsf{f_G}}=-\left[b_\zeta,a_\xi \right]_{\mathsf{f_G}}$, $\left[a_\xi,b_\zeta+c_\zeta \right]_{\mathsf{f_G}}=\left[a_\xi,b_\zeta \right]_{\mathsf{f_G}}+\left[a_\xi,c_\zeta \right]_{\mathsf{f_G}}$, $\left[a_\xi+d_\xi,b_\zeta \right]_{\mathsf{f_G}}=\left[a_\xi,b_\zeta \right]_{\mathsf{f_G}}+\left[d_\xi,b_\zeta \right]_{\mathsf{f_G}}$, and $\left[a_\xi,\lambda b_\zeta \right]_{\mathsf{f_G}}=\lambda\left[a_\xi,b_\zeta \right]_{\mathsf{f_G}}=\left[\lambda a_\xi, b_\zeta \right]_{\mathsf{f_G}}$, for any $a_\xi,d_\xi\in\mathfrak{A}_\xi$, $b_\zeta,c_\zeta\in\mathfrak{A}_\zeta$, and $\lambda\in\mathbb{F}$. 
As the decomposition of each element of $\mathfrak{A}$ in homogeneous elements is unique, given any $a=\sum_{\xi\in\mathsf{G}} a_{\xi}$ and $b=\sum_{\zeta\in\mathsf{G}} b_{\zeta}$ in $\mathfrak{A}$, we define the $\mathsf{f_G}$-commutator of $a$ and $b$ as 
\begin{equation*}
[a,b]_{\mathsf{f_G}}=\left[\sum_{\xi\in\mathsf{G}} a_\xi,\sum_{\zeta\in\mathsf{G}}b_\zeta\right]_{\mathsf{f_G}}\coloneqq\sum_{\xi,\zeta\in\mathsf{G}} \left[a_\xi,b_\zeta \right]_{\mathsf{f_G}}\ .
\end{equation*}
Obviously $\left[a,b \right]_{\mathsf{f_G}}=-\left[b,a \right]_{\mathsf{f_G}}$, $\left[a,b+c \right]_{\mathsf{f_G}}=\left[a,b \right]_{\mathsf{f_G}}+\left[a,c \right]_{\mathsf{f_G}}$ and $\left[\lambda a,b \right]_{\mathsf{f_G}}=\left[a,\lambda b \right]_{\mathsf{f_G}}=\lambda \left[a,b \right]_{\mathsf{f_G}}$ for any $a,b,c\in\mathfrak{A}$ and $\lambda \in\mathbb{F}$. For each $\mathsf{f_G}$, we say only ``\textit{$[\ ,\ ]_{\mathsf{f_G}}$ is an $\mathsf{f_G}$-commutator of $\mathfrak{A}$ which depends of  $\Gamma$}''. And if $[a,b]_{\mathsf{f_G}}=0$ for any $a,b\in\mathfrak{A}$, we say that ``\textit{$\mathfrak{A}$ is an $\mathsf{f_G}$-commutative algebra}''.

When $[x^{(\xi)},y^{(\zeta)}]_{\mathsf{f_G}}$ is a nontrivial graded polynomial of $\mathbb{F}\langle X^{\mathsf{G}}\rangle$, there are 3 possibilities for $\mathsf{f_G}$ in $(\xi,\zeta)$: \textit{i)} $\mathsf{f_G}(\xi,\zeta)=0$ and $\mathsf{f_G}(\zeta,\xi)\neq0$; \textit{ii)} $\mathsf{f_G}(\xi,\zeta)\neq0$ and $\mathsf{f_G}(\zeta,\xi)=0$; \textit{iii)} $\mathsf{f_G}(\xi,\zeta)\neq0$ and $\mathsf{f_G}(\zeta,\xi)\neq0$. The Examples \ref{3.17}, \ref{3.16} and \ref{3.18} below illustrate graded algebras that satisfy graded polynomial identities which represent each of these three cases. When $\mathsf{f_G}(\xi,\zeta)\in\mathbb{F}^*$ for any $\xi, \zeta\in\mathsf{G}$, note that $[x^{(\xi)},y^{(\zeta)}]_{\mathsf{f_G}}\equiv_{\mathsf{G}}0$ for any $\xi, \zeta\in\mathsf{G}$ implies that the homogeneous elements of $\mathfrak{A}$ are $\mathsf{t}$-commutative (as in Definition 2.5, in \cite{Mardua01}) with each other, for $\mathsf{t}$ defined by $\mathsf{t}(a_\xi,b_\zeta)=\mathsf{f_G}(\xi,\zeta)^{-1}\mathsf{f_G}(\zeta,\xi)$ for any homogeneous elements $a_\xi\in\mathfrak{A}_\xi$ and $b_\zeta\in\mathfrak{A}_\zeta$. Consequently, for any map  $\mathsf{f_G}: \mathsf{G}\times \mathsf{G} \rightarrow \mathbb{F}$, it follows that
	\begin{equation*}
	\left[x^{(\xi)}, y^{(\zeta)} \right]_{\mathsf{f_{G}}}=\left\{
		\begin{array}{c l}	
    \mathsf{f_G}(\xi,\zeta) \left[x^{(\xi)} \ , y^{(\zeta)}\right]_{\mathsf{t}} \ ,& \mbox{if } \mathsf{f_G}(\xi,\zeta)\neq0 \mbox{ and } \mathsf{f_G}(\zeta,\xi)\neq0 \\
     \mathsf{f_G}(\xi,\zeta) x^{(\xi)}  y^{(\zeta)} \ ,& \mbox{if } \mathsf{f_G}(\xi,\zeta)\neq0 \mbox{ and } \mathsf{f_G}(\zeta,\xi)=0 \\
    - \mathsf{f_G}(\zeta,\xi) y^{(\zeta)} x^{(\xi)}  \ ,& \mbox{if } \mathsf{f_G}(\xi,\zeta)=0 \mbox{ and } \mathsf{f_G}(\zeta,\xi)\neq0\\
     0  \ ,& \mbox{if } \mathsf{f_G}(\xi,\zeta)=\mathsf{f_G}(\zeta,\xi)=0
		\end{array}\right. \ .
	\end{equation*}
Conversely, if $\mathfrak{A}$ is $\mathsf{h}$-commutative for some map $\mathsf{h}$ from $\mathfrak{A}\times\mathfrak{A}$ in $\mathbb{F}$ which satisfies $\mathsf{h}(a_\xi, b_\zeta)=\mathsf{h}(d_\xi,c_\zeta)$ for any $a_\xi,d_\xi\in\mathfrak{A}_\xi\setminus\{0\}$ and $b_\zeta,c_\zeta\in\mathfrak{A}_\zeta\setminus\{0\}$, the map $\mathsf{h_G}$ defined by $\mathsf{h_G}(\xi,\zeta)=1+\mathsf{h}(b_\zeta,a_\xi)$, for any $a_\xi\in\mathfrak{A}_\xi$ and $b_\zeta\in\mathfrak{A}_\zeta$, makes $\mathfrak{A}$ an $[\ ,\ ]_{\mathsf{h_G}}$-commutative algebra.

\begin{example}\label{3.17}
Let $\mathbb{F}$ be a field, $\mathcal{K}=\mathbb{Z}_2\times\mathbb{Z}_2$ the Klein group, and $\mathfrak{B}=M_2(\mathbb{F})$ the $2\times 2$ matrix algebra over $\mathbb{F}$ with its natural $\mathcal{K}$-grading, i.e. $\mathfrak{B}_{(\bar0,\bar0)}=\mathsf{span}_{\mathbb{F}}\{E_{11}+E_{22}\}$, $\mathfrak{B}_{(\bar1,\bar1)}=\mathsf{span}_{\mathbb{F}}\{E_{11}-E_{22}\}$, $\mathfrak{B}_{(\bar0,\bar1)}=\mathsf{span}_{\mathbb{F}}\{E_{12}+E_{21}\}$ and $\mathfrak{B}_{(\bar1,\bar0)}=\mathsf{span}_{\mathbb{F}}\{E_{12}-E_{21}\}$, where $E_{ij}$'s are the elementary matrices. 
Note that $x^{(\xi)}y^{(\zeta)}+y^{(\zeta)}x^{(\xi)}\equiv_{\mathcal{K}}0$ in $\mathfrak{B}$ for distinct $\xi,\zeta\in\mathcal{K}\setminus\{(\bar0,\bar0)\}$, and $(z^{(\tau)})^2\not\equiv_{\mathcal{K}}0$ in $\mathfrak{B}$ for any $\tau\in\mathcal{K}$. On the other side, $[x^{((\bar0,\bar0))},y^{(\xi)}]\equiv_{\mathcal{K}}0$ and $[y^{(\xi)},z^{(\xi)}]\equiv_{\mathcal{K}}0$ in $\mathfrak{B}$, for any $\xi\in\mathcal{K}$. Define the map $\mathsf{f}_{\mathcal{K}}$ from $\mathcal{K}\times\mathcal{K}$ in $\mathbb{F}$ satisfying $\mathsf{f}_{\mathcal{K}}((\bar0,\bar0),\xi)=\mathsf{f}_{\mathcal{K}}(\xi,(\bar0,\bar0))=\mathsf{f}_{\mathcal{K}}(\xi,\xi)=1$ for any $\xi\in\mathcal{K}$ and $\mathsf{f}_{\mathcal{K}}(\zeta,\tau)=-\mathsf{f}_{\mathcal{K}}(\tau,\zeta)$ for distinct $\zeta,\tau\in\mathcal{K}\setminus\{(\bar0,\bar0)\}$. Therefore, $\mathfrak{B}$ is $[\ ,\ ]_{\mathsf{f}_{\mathcal{K}}}$-commutative.
\end{example}

For the two examples below, being $\mathfrak{B}=M_{k}(\mathbb{F})$, $\mathsf{G}$ any group and $\theta$ a $k$-tuple in $\mathsf{G}^k$, the ``\textit{elementary $\mathsf{G}$-grading on $\mathfrak{B}$ defined by $\theta$}'' is the $\mathsf{G}$-grading on $\mathfrak{B}$ defined by $\mathfrak{B}_\xi=\mathsf{span}_{\mathbb{F}}\{E_{ij}\in\mathfrak{B}: \theta_i^{-1}\theta_j=\xi\}$.

\begin{example}\label{3.16}
Let $\mathbb{F}$ be a field, $\mathsf{G}=\mathbb{Z}_{3}\times\mathbb{Z}_{5}$ the cyclic group of order 15, and $\mathfrak{B}=M_4(\mathbb{F})$ the $4\times 4$ matrix algebra over $\mathbb{F}$. Consider the elementary $\mathsf{G}$-grading $\Gamma$ on $\mathfrak{B}$ defined by the $4$-tuple $\theta=(\theta_1,\theta_2,\theta_3,\theta_4)\in\mathsf{G}^4$, where $\theta_1=(\bar0,\bar0)$, $\theta_2=(\bar1,\bar0)$, $\theta_3=(\bar0,\bar1)$ and $\theta_4=(\bar0,\bar4)$. We have that $\mathsf{Supp}(\Gamma)=\mathsf{G}\setminus\{(\bar1,\bar2), (\bar1,\bar3), (\bar2,\bar2), (\bar2,\bar3)\}$, and 
	\begin{equation*}
	\begin{array}{l l l}	
\mathfrak{B}_{(\bar0,\bar0)}=\mathsf{span}_{\mathbb{F}}\{E_{11},E_{22},E_{33},E_{44}\}, &
\mathfrak{B}_{(\bar0,\bar4)}=\mathsf{span}_{\mathbb{F}}\{E_{14},E_{31}\}, &
\mathfrak{B}_{(\bar2,\bar0)}=\mathsf{span}_{\mathbb{F}}\{E_{21}\}, \\

\mathfrak{B}_{(\bar0,\bar1)}=\mathsf{span}_{\mathbb{F}}\{E_{13},E_{41}\}, &
\mathfrak{B}_{(\bar1,\bar0)}=\mathsf{span}_{\mathbb{F}}\{E_{12}\}, &
\mathfrak{B}_{(\bar2,\bar1)}=\mathsf{span}_{\mathbb{F}}\{E_{23}\}, \\

\mathfrak{B}_{(\bar0,\bar2)}=\mathsf{span}_{\mathbb{F}}\{E_{43}\}, &
\mathfrak{B}_{(\bar1,\bar1)}=\mathsf{span}_{\mathbb{F}}\{E_{42}\}, &
\mathfrak{B}_{(\bar2,\bar4)}=\mathsf{span}_{\mathbb{F}}\{E_{24}\},\\

\mathfrak{B}_{(\bar0,\bar3)}=\mathsf{span}_{\mathbb{F}}\{E_{34}\}, &
\mathfrak{B}_{(\bar1,\bar4)}=\mathsf{span}_{\mathbb{F}}\{E_{32}\}. &
	\end{array} 
	\end{equation*}
Consider the map $\mathsf{f_G}:\mathsf{G}\times\mathsf{G}\rightarrow\mathbb{F}$ satisfying $\mathsf{f_G}((\bar0,\bar0),(\bar0,\bar0))=\mathsf{f_G}((\bar1,\bar0),(\bar0,\bar2))=\mathsf{f_G}((\bar0,\bar2),(\bar1,\bar0))=\mathsf{f_G}((\bar2,\bar0),(\bar1,\bar1))=\mathsf{f_G}((\bar2,\bar1),(\bar1,\bar1))=1$ and $\mathsf{f_G}(\xi,\zeta)=0$ for any $(\xi,\zeta)\in\mathsf{G}\times\mathsf{G}$ such that $(\xi,\zeta)\notin\{((\bar0,\bar0),(\bar0,\bar0)),((\bar1,\bar0),(\bar0,\bar2)),((\bar0,\bar2),(\bar1,\bar0)),((\bar2,\bar0),(\bar1,\bar1)),((\bar2,\bar1),(\bar1,\bar1))\}$. Therefore, we conclude that $g(x_1^{(\xi_1)},\dots,x_{15}^{(\xi_{15})})=\sum_{i,j=1}^{15}[x_i^{(\xi_i)},x_j^{(\xi_j)}]_{\mathsf{f_G}}+\sum_{\xi_l\notin\{\theta_1,\theta_3,\theta_4\}}(x_l^{(\xi_l)})^2\equiv_{\mathsf{G}}0$ in $\mathfrak{B}$.
\end{example}

\begin{example}\label{3.18}
Consider $\mathsf{G}$ and $\mathfrak{B}$ as in Example \ref{3.16}, $\mathfrak{B}$ with the elementary $\mathsf{G}$-grading defined by the $4$-tuple $\hat\theta=(\hat\theta_1,\hat\theta_2,\hat\theta_3,\hat\theta_4)\in\mathsf{G}^4$, where $\hat\theta_1=(\bar0,\bar0)$, $\hat\theta_2=(\bar1,\bar0)$ and $\hat\theta_3=\hat\theta_4=(\bar0,\bar1)$. As $\mathfrak{B}_{(\bar1,\bar0)}=\mathsf{span}_{\mathbb{F}}\{E_{12}\}$ and $\mathfrak{B}_{(\bar1,\bar4)}=\mathsf{span}_{\mathbb{F}}\{E_{32},E_{42}\}$, we have that $x^{(\bar1,\bar0)}y^{(\bar1,\bar4)}\equiv_{\mathsf{G}}0$, $y^{(\bar1,\bar4)}x^{(\bar1,\bar0)}\equiv_{\mathsf{G}}0$, $(x^{(\bar1,\bar4)})^2\equiv_{\mathsf{G}}0$ and $(y^{(\bar1,\bar0)})^2\equiv_{\mathsf{G}}0$ in $\mathfrak{B}$. On the other hand, putting $e=(\bar0,\bar0)$, we have that $[x^{(e)},y^{(e)}]$ is not a graded polynomial identity for $\mathfrak{B}$, because $E_{34}, E_{43}\in\mathfrak{B}_{e}$. Thus, if $\mathfrak{B}$ is $[\ ,\ ]_{\mathsf{f_G}}$-commutative for some map $\mathsf{f_G}$, then $\mathsf{f_G}(\xi,\zeta)$ is not necessarily zero for any $\xi,\zeta\in\{(\bar1,\bar0),(\bar1,\bar4)\}$, but we must have $\mathsf{f_G}(e,e)=0$.
\end{example}

%
%
%
\subsection{Some Results in Graded Algebras and PI-Theory}
Here, let us review key results from graded algebras and PI-Theory that will be used in the next section. 
The first result is due to J. Bergen and M. Cohen (1986). Posteriorly, Yu. Bahturin, A. Giambruno and D. Riley (1998) showed the same result and, in addition, presented bounds for the degrees of the polynomial identities involved.

\begin{lemma}[Corollary 9 in \cite{BergCohe86}, or Theorem 5.3 in \cite{BahtGiamRile98}]\label{teoBergCoheBaht}
	Let $\mathsf{G}$ be a finite group with neutral element $e$, and $\mathfrak{A}$ a $\mathsf{G}$-graded algebra. If $\mathfrak{A}_e$ is a $PI$-algebra, then $\mathfrak{A}$ is also a $PI$-algebra.
\end{lemma}

Inspired by this result, in 2022, A. de Fran\c{c}a and I. Sviridova proved in \cite{Mardua01} the following results:
\begin{lemma}[Theorem 3.9, \cite{Mardua01}]\label{mardua3.9}
	Let $\mathsf{S}$ be a left cancellative monoid and $\mathfrak{R}$ a ring with a finite $\mathsf{S}$-grading of order $d$. If $\mathfrak{R}_e$ is nilpotent of index $\mathsf{nd}(\mathfrak{R}_e)=r\geq1$, then $\mathfrak{R}$ is a nilpotent ring, such that $r \leq \mathsf{nd}(\mathfrak{R})\leq dr$ for $r>1$, and $r\leq\mathsf{nd}(\mathfrak{R})\leq d+1$ for $r=1$.
\end{lemma}

\begin{lemma}[Proposition 4.7, \cite{Mardua01}]\label{mardua4.7}
	Let $\mathsf{S}$ be a monoid and $\mathfrak{R}$ a ring with a finite $\mathsf{S}$-grading of order $d$. If $\mathfrak{R}_e$ is nil of index $2$ and $\mathsf{char}(\mathfrak{R}_e)\neq2$, then $\mathfrak{R}$ is nilpotent with $\mathsf{nd}(\mathfrak{R})\leq 3d$.
\end{lemma}

In the two results below, $\mathbb{F}$ is an algebraically closed field of characteristic zero. Recall that a $\mathsf{G}$-graded algebra $\mathfrak{A}$ is said to be \textit{simple graded} (or \textit{$\mathsf{G}$-simple}) if $\mathfrak{A}^2\neq\{0\}$ and $\mathfrak{A}$ does not have proper $\mathsf{G}$-graded ideals. In \cite{BahtSehgZaic08}, Yu. Bahturin, M. Zaicev and S. Sehgal classified the $\mathsf{G}$-simple $\mathbb{F}$-algebras of finite dimension. They proved the following result:
\begin{lemma}[Theorem 3, \cite{BahtSehgZaic08}]\label{Bahtteo3}
	Let $\mathsf{G}$ be any group, and $\mathfrak{A}$ a finite dimensional $\mathsf{G}$-graded $\mathbb{F}$-algebra. Then $\mathfrak{A}$ is $\mathsf{G}$-simple iff $\mathfrak{A}$ is $\mathsf{G}$-isomorphic to $\mathfrak{B}=M_{k}(\mathbb{F}^{\sigma}[\mathsf{H}])$, where $\mathsf{H}$ is a finite subgroup of $\mathsf{G}$ and $\sigma: \mathsf{H}\times\mathsf{H}\rightarrow\mathbb{F}^*$ is a $2$-cocycle on $\mathsf{H}$. The $\mathsf{G}$-grading on $\mathfrak{B}$ is the defined by a $k$-tuple $(\theta_1,\dots,\theta_k)\in\mathsf{G}^k$ so that $\mathfrak{B}_\xi=\mathsf{span}_{\mathbb{F}}\{E_{ij}\eta_\zeta: \theta_i^{-1}\zeta \theta_j=\xi\}$.
\end{lemma}

We call ``\textit{the elementary-canonical $\mathsf{G}$-grading on $\mathfrak{B}$ defined by $\theta$}'' the $\mathsf{G}$-grading on $\mathfrak{B}=M_{k}(\mathbb{F}^{\sigma}[\mathsf{H}])$ defined in Lemma \ref{Bahtteo3}, where $\theta=(\theta_1,\dots,\theta_k)$ is a $k$-tuple of $\mathsf{G}^k$.

In \cite{Svir11}, a graded version of Wedderburn-Malcev Theorem was presented by I. Sviridova. She showed the following:
\begin{lemma}[Lemma 2, \cite{Svir11}]\label{teoIrina03}
	Let $\mathsf{G}$ be any finite abelian group. Any finite dimensional $\mathsf{G}$-graded $\mathbb{F}$-algebra $\mathfrak{A}$ is isomorphic as $\mathsf{G}$-graded algebra to a $\mathsf{G}$-graded $\mathbb{F}$-algebra of the form 
	\begin{equation}\label{3.08}
		\mathfrak{A}' = \left(M_{k_1}(\mathbb{F}^{\sigma_1}[\mathsf{H}_1]) \times \cdots \times M_{k_p}(\mathbb{F}^{\sigma_p}[\mathsf{H}_p])\right) \oplus \mathsf{J} \ .
	\end{equation}
	Here the Jacobson radical $\mathsf{J} = \mathsf{J}(\mathfrak{A})$ of $\mathfrak{A}$ is a graded ideal, and $\mathfrak{B} = M_{k_1}(\mathbb{F}^{\sigma_1}[\mathsf{H}_1]) \times \cdots \times M_{k_p}(\mathbb{F}^{\sigma_p}[\mathsf{H}_p])$ (direct product of algebras) is the maximal graded semisimple subalgebra of $\mathfrak{A}'$, $p\in\mathbb{N}\cup\{0\}$. The $\mathsf{G}$-grading on $\mathfrak{B}_l = M_{k_l} (\mathbb{F}^{\sigma_l} [\mathsf{H}_l])$ is the elementary-canonical grading corresponding to some $k_l$-tuple $(\theta_{l_1},\dots,\theta_{l_{k_l}})\in \mathsf{G}^{k_l}$, where $\mathsf{H}_l$ is a subgroup $\mathsf{G}$ and $\sigma\in\mathsf{Z}^2(\mathsf{H}_l,\mathbb{F}^*)$ is a $2$-cocycle.
\end{lemma}

Finally, in PI-Theory, an important problem is the well-known \textit{Specht Problem}. Originally posed by W. Specht (1950), in \cite{Spec50}, the Specht Problem asks whether \textit{any set of polynomial identities of a given algebra $\mathfrak{A}$ is a consequence of a finite number of identities of $\mathfrak{A}$}. Posteriorly, in \cite{Keme91}, A. Kemer (1991) showed that the Specht Problem has a positive solution in the variety of associative algebras of characteristic zero. Recall that \textit{the variety} (\textit{of associative $\mathbb{F}$-algebras}) defined by the system of polynomial identities $\{f_i: i\in I\}\subset\mathbb{F}\langle X \rangle$ is the class $\mathfrak{W}$ of all associative $\mathbb{F}$-algebras satisfying all the $f_i$'s, $i\in I$. For further reading on the Specht Problem, see works \cite{BeloRoweVish12}, \cite{Geno81}, \cite{Keme87,Keme91} and \cite{Popo81}. Additionally, for more details on varieties of algebras, see \cite{Dren00}, Chapter 2, or \cite{GiamZaic05}, Chapter 1. 

The next two results, due to I. Sviridova, provide a positive answer to graded version of Specht Problem. In both results, $\mathbb{F}$ is an algebraically closed field of characteristic zero, and $\mathsf{G}$ is any finite abelian group. 

\begin{lemma}[Theorem 1, \cite{Svir11}]\label{teoIrina01}
	Any $\mathsf{G} T$-ideal of $\mathsf{G}$-graded identities of a finitely generated associative $PI$-algebra over $\mathbb{F}$ graded by $\mathsf{G}$ coincides with the ideal of $\mathsf{G}$-graded identities of some finite dimensional associative $\mathsf{G}$-graded $\mathbb{F}$-algebra.
\end{lemma}

\begin{lemma}[Theorem 2, \cite{Svir11}]\label{teoIrina02}
	Any $\mathsf{G} T$-ideal of graded identities of a $\mathsf{G}$-graded associative $PI$-algebra over $\mathbb{F}$ coincides with the ideal of $\mathsf{G}$-graded identities of the $\mathsf{G}$-graded Grassmann envelope of some finite dimensional over $\mathbb{F}$ associative $\mathsf{G} \times \mathbb{Z}_2$-graded algebra.
\end{lemma}

Recall that \textit{Grassmann Envelope} of a $(\mathsf{G}\times\mathbb{Z}_2)$-graded $\mathbb{F}$-algebra $\mathfrak{A}$, denoted by $\mathsf{E}^\mathsf{G}(\mathfrak{A})$, is defined by 
\begin{equation*}
	\mathsf{E}^\mathsf{G}(\mathfrak{A})=\left(\mathfrak{A}_0\otimes \mathsf{E}_0 \right) \oplus \left(\mathfrak{A}_1\otimes \mathsf{E}_1 \right) ,
\end{equation*}
where $\mathsf{E}=\mathsf{E}_0\oplus \mathsf{E}_1$ is an infinitely generated non-unitary Grassmann algebra with its natural $\mathbb{Z}_2$-grading, i.e.  $\mathsf{E}$ is the $\mathbb{F}$-algebra generated by elements $e_1,e_2,e_3,\dots$, such that $e_ie_j=-e_je_i$, for all $i,j\in\mathbb{N}$, where $\mathsf{char}(\mathbb{F})=0$, and the $\mathbb{Z}_2$-grading on $\mathsf{E}$ is given by $\mathsf{E}_0=\mathsf{span}_\mathbb{F}\{e_{i_1}e_{i_2}\cdots e_{i_n}: n \mbox{ is even}\}$, and $\mathsf{E}_1=\mathsf{span}_\mathbb{F}\{e_{j_1}e_{j_2}\cdots e_{j_m}: m \mbox{ is odd}\}$.

It is worth noting that, in \cite{AljaBelo10}, E. Aljadeff and A. Kanel-Belov (2010) showed, independently to I. Sviridova, a result similar to the above, without requiring the group $\mathsf{G}$ to be abelian. They proved that ``\textit{if $\mathsf{G}$ is a finite group and $W$ is a $\mathsf{G} PI$-graded algebra over $\mathbb{F}$, $\mathsf{char}(\mathbb{F})=0$, then there is a field extension $\mathbb{K}$ of $\mathbb{F}$ and a finite-dimensional $(\mathsf{G}\times \mathbb{Z}_2)$-graded algebra $\mathfrak{A}$ over $\mathbb{K}$ such that $\mathsf{T}^\mathsf{G}(W)=\mathsf{T}^\mathsf{G}(\mathsf{E}^\mathsf{G}(\mathfrak{A}))$}''.

%
%
\subsection{First Results}
Let us begin the study of the graded algebras which satisfy some graded polynomial identity of degree $2$. A polynomial $g=g(x_1, \dots, x_n)$ of $\mathbb{F}\langle X \rangle$ has degree $2$ when it is of the form 
\begin{equation*}
	\begin{split}
g(x_1, \dots, x_n)=\sum_{r,s, k=1}^n \lambda_{rs}x_r x_s + \gamma_k x_k \ ,
	\end{split}
\end{equation*}
where $\lambda_{rs}, \gamma_k\in\mathbb{F}$, with $\lambda_{rs}\neq0$ for some $r,s$. Analogously a graded polynomial $f=f(y_1^{(\xi_1)}, \dots, y_n^{(\xi_m)})$ of $\mathbb{F}\langle X^{\mathsf{G}} \rangle$ has degree $2$ when it is of the form 
\begin{equation*}
	\begin{split}
f(y_1^{(\xi_1)}, \dots, y_m^{(\xi_m)})=\sum_{r,s, k=1}^m \delta_{rs}y_r^{(\xi_r)} y_s^{(\xi_s)} + \theta_k y_k^{(\xi_k)} \ ,
	\end{split}
\end{equation*}
where $\delta_{rs}, \theta_k\in\mathbb{F}$, with $\delta_{rs}\neq0$ for some $r,s$.

\begin{lemma}\label{3.13}
Let $\mathsf{G}$ be a group, $\mathbb{F}$ a field with $|\mathbb{F}|>2$, and $\mathfrak{A}$ an $\mathbb{F}$-algebra with a $\mathsf{G}$-grading $\Gamma$. If $\mathfrak{A}$ satisfies a graded polynomial identity $g=g(x^{(\xi_1)}_1, \dots, x^{(\xi_n)}_n)\in\mathbb{F}\langle X^{\mathsf{G}} \rangle$ of degree $2$, then $\mathfrak{A}$ satisfies a graded polynomial identity of degree $2$ of the form 
	\begin{equation}\label{3.02}
	\begin{split}
\widehat{g}(x^{(\xi_1)}_1, \dots, x^{(\xi_n)}_n)=\sum_{1\leq r<s\leq n} \gamma_{rs} \left[x^{(\xi_r)}_r, x^{(\xi_s)}_s \right]_{\mathsf{f_{G}}}
+\sum_{1\leq k \leq n} \delta_{k}\left(x_k^{(\xi_k)} \right)^2 \ ,
	\end{split}
	\end{equation}
where $\gamma_{rs},\delta_{k}\in\mathbb{F}$ and $[\ , \ ]_{\mathsf{f_G}}$ is an $\mathsf{f_{G}}$-commutator of $\mathfrak{A}$ which depends of $\Gamma$.
\end{lemma}
\begin{proof}
Put 
$g(x^{(\xi_1)}_1, \dots, x^{(\xi_n)}_n)=\sum_{r,s, k=1}^n \lambda_{rs}x_r^{(\xi_r)} x_s^{(\xi_s)} + \gamma_k x_k^{(\xi_k)} $, 
for some  $\lambda_{rs}, \gamma_k\in\mathbb{F}$. Take any $k\in\{1,\dots, n\}$. Let us first show that either $\gamma_k=0$ or $\xi_k\notin\mathsf{Supp}(\Gamma)$. Assume that $\xi_k\in\mathsf{Supp}(\Gamma)$, and take any nonzero $a\in\mathfrak{A}_{\xi_k}$. In $g$, replacing $x^{(\xi_k)}_k$ by $a$ and $x^{(\xi_s)}_s$ by $0$ when $s\neq k$, we have 
	\begin{equation}\label{3.03}
0=g(0,\dots,0, a,0, \dots, 0)=\lambda_{kk} a^2 + \gamma_k  a \ .
	\end{equation}
So, $\lambda_{kk} a^2 + \gamma_k  a=0$ for any nonzero $a\in \mathfrak{A}_{\xi_k}$. If $\xi_k\neq e$, then $\xi_k^2\neq \xi_k$, and hence, $\lambda_{kk} a^2 = \gamma_k a=0$, for any $a\in \mathfrak{A}_{\xi_k}\setminus\{0\}$. Consequently, $\gamma_k=0$. On the other hand, suppose $\xi_k= e$. Thus, $\lambda a, (\lambda a)^2\in \mathfrak{A}_e$ for any $\lambda\in\mathbb{F}^*$ and $a\in\mathfrak{A}_e$. Hence, by \eqref{3.03}, we have that $0=\lambda_{kk} (\lambda a)^2 + \gamma_k  (\lambda a)=\lambda_{kk} \lambda^2 a^2 + \gamma_k \lambda  a$, and so $\gamma_k a=-\lambda \lambda_{kk} a^2$, for any $\lambda\in\mathbb{F}^*$ and $a\in\mathfrak{A}_e$, $a\neq0$. Because $|\mathbb{F}|>2$, we must have $\gamma_k=0$, otherwise $a=-\lambda(\gamma_k^{-1}\lambda_{kk}a^2)$ for any $\lambda\in\mathbb{F}^*$ and $a\in\mathfrak{A}_e\setminus\{0\}$, which leads to a contradiction. From this, we conclude that $\gamma_k=0$ for all $k=1,\dots,n$. 

Consider the following graded polynomial
	\begin{equation*}
	\widehat{g}(x^{(\xi_1)}_1, \dots, x^{(\xi_n)}_n)=\sum_{1\leq r<s\leq n} \left(\lambda_{rs}x_r^{(\xi_r)} x_s^{(\xi_s)}+\lambda_{sr} x_s^{(\xi_s)} x_r^{(\xi_r)}\right) + \sum_{k=1}^n \delta_{k}\left(x_k^{(\xi_k)} \right)^2\ ,
	\end{equation*}
where $\delta_k=\lambda_{kk}$ for all $k=1,\dots,n$. \textbf{Claim:} $\widehat{g}\equiv_{\mathsf{G}}0$ in $\mathfrak{A}$. In fact, first, by \eqref{3.03}, $\delta_k a^2=\lambda_{kk} a^2=0$ for any $k=1,\dots,n$ and $a\in \mathfrak{A}_{\xi_k}$. Now, fixed any $r,s\in\{1,\dots,n\}$, with $r<s$, take  $a_{\xi_r}\in\mathfrak{A}_{\xi_r}$ and  $b_{\xi_s}\in\mathfrak{A}_{\xi_s}$. In $\widehat{g}$, replacing $x^{(\xi_r)}_r$ by $a_{\xi_r}$, $x^{(\xi_s)}_s$ by $b_{\xi_s}$, and $x^{(\xi_k)}_k$ by $0$ when $k\notin\{s, r\}$, we have that
	\begin{equation*}
0=\widehat{g}(0,\dots,0,a_{\xi_r},0, \dots, 0,b_{\xi_s},0, \dots, 0)=\lambda_{rs}a_{\xi_r} b_{\xi_s}+\lambda_{sr} b_{\xi_s} a_{\xi_r}+\delta_{r}a_{\xi_r}^2+\delta_{s} b_{\xi_s}^2 \ .
	\end{equation*}
By the above, we deduce that $\delta_{r}a_{\xi_r}^2=\delta_{s}b_{\xi_s}^2=0$, and so $\lambda_{rs}a_{\xi_r} b_{\xi_s}+\lambda_{sr} b_{\xi_s} a_{\xi_r}=0$. The claim follows. 

 Suppose $\xi_r=\xi_s\in\mathsf{Supp}(\Gamma)$ for $r\neq s$. Observe that if $x_r^{(\xi_r)}x_s^{(\xi_s)}\equiv_{\mathsf{G}}0$ in $\mathfrak{A}$, then we can assume $\delta_r\neq0$ (or $\delta_r\neq0$) in $\widehat{g}$. Conversely, if $(x_r^{(\xi_r)})^2\not\equiv_{\mathsf{G}}0$ in $\mathfrak{A}$, then we must have $\lambda_{rs}=-\lambda_{sr}$ (and $\lambda_{rr}=\lambda_{ss}=0$), and hence, $\lambda_{rs}(x_r^{(\xi_r)} x_s^{(\xi_s)}- x_s^{(\xi_s)} x_r^{(\xi_r)})\equiv_{\mathsf{G}}0$ in $\mathfrak{A}$.

Finally, let $\{{i_1},\dots,{i_m}\}$ be the smallest subset of $\{1,\dots,n\}$ such that $\{\xi_{i_1},\dots,\xi_{i_m}\} = \{\xi_1,\dots,\xi_n\}$ (and so $\xi_{i_r}\neq\xi_{i_s}$ for $r\neq s$), and define the map $\mathsf{f_G}$ from $\mathsf{G}\times\mathsf{G}$ in $\mathbb{F}$ as follows:
\begin{equation*}
	\mathsf{f_G}(\xi,\zeta)=\left\{
	\begin{array}{c l}	
		\lambda_{i_r i_s} \ ,& \mbox{if } (\xi,\zeta)=(\xi_{i_r},\xi_{i_s}) \mbox{ with } i_r<i_s \\
		-\lambda_{i_s i_r} \ ,& \mbox{if } (\xi,\zeta)=(\xi_{i_s},\xi_{i_r}) \mbox{ with } i_s < i_r \\
		\lambda_{rs} \ , & \mbox{if } \xi=\zeta=\xi_r=\xi_s \mbox{ and } (x^{(\xi)})^2\not\equiv_{\mathsf{G}}0 \mbox{ in } \mathfrak{A} \\
		0 \ , & \mbox{if either } \xi=\zeta \mbox{ and }  (x^{(\xi)})^2\equiv_{\mathsf{G}}0 \mbox{ in } \mathfrak{A} \\
			  &  \hspace{3.2mm} \mbox{or } \{\xi,\zeta\} \not\subset \{\xi_1,\dots,\xi_n\}
	\end{array}\right. \ .
\end{equation*}
Hence, the polynomial $\displaystyle\widetilde{g}(x^{(\xi_1)}_1, \dots, x^{(\xi_n)}_n)= \sum_{1\leq r<s\leq n} \left[x_r^{(\xi_r)} , x_s^{(\xi_s)}\right]_{\mathsf{f_{G}}} + \sum_{1\leq k \leq n} \delta_{k}\left(x_k^{(\xi_k)} \right)^2$ is a graded polynomial identity for $\mathfrak{A}$. The result follows.
\end{proof}

In the proof of the previous lemma, observe that the polynomial $[x^{(\xi_r)},y^{(\xi_s)}]_{\mathsf{f_{G}}}$ is also zero when $\lambda_{rs}=\lambda_{sr}=0$ for some $r,s\in\{1,\dots,n\}$. Furthermore, when $\xi_i=\xi_j=e$ and $\lambda_{ij}=\lambda_{ji}\neq0$ for some $i,j\in\{1,\dots,n\}$, and $\mathfrak{A}_e$ is not nil of index $2$ (see Lemma \ref{mardua4.7}), we can conclude that $[x^{(e)},y^{(e)}]_{\mathsf{f_G}}=\lambda[x^{(e)},y^{(e)}]$ and $[x^{(e)},y^{(e)}]\equiv_{\mathsf{G}}0$ in $\mathfrak{A}$, where $\lambda=\lambda_{ij}$.

\begin{remark}
Throughout this text, from now on, let us assume that, for any graded polynomial identity $g=g(x^{(\xi_1)}_1, \dots, x^{(\xi_n)}_n)\in\mathbb{F}\langle X^{\mathsf{G}} \rangle$ of degree $2$ of a given algebra $\mathfrak{A}$ with a $\mathsf{G}$-grading $\Gamma$, the $\xi_r$'s belong to support of $\Gamma$. So let us assume also that not all $\gamma_{rs}$'s or $\delta_k$'s in \eqref{3.02}, in Lemma \ref{3.13}, are zero.
\end{remark}

By Lemma \ref{mardua4.7}, a ring $\mathfrak{R}$ with a finite $\mathsf{G}$-grading is nilpotent when its neutral component $\mathfrak{R}_e$ is nil of index $2$ and $\mathsf{char}(\mathfrak{R}_e)\neq2$. Hence, under the same assumptions of Lemma \ref{3.13}, and adding the conditions ``$\mathfrak{A}$ has finite grading'' and ``$\mathsf{char}(\mathbb{F})\neq2$'', if $\mathfrak{A}$ is not nilpotent, then $b^2\neq0$ for some $b\in\mathfrak{A}_e$, and hence, the polynomial $\widehat{g}$ in \eqref{3.02} can be rewritten as 
\begin{equation*}
	\begin{split}
\widehat{g}(x^{(\xi_1)}_1, \dots, x^{(\xi_n)}_n)=\sum_{1\leq r<s\leq n} \gamma_{rs} \left[x^{(\xi_r)}_r, x^{(\xi_s)}_s \right]_{\mathsf{f_G}}
+\sum_{\xi_k\in\mathsf{G}\setminus \{e\}} \delta_{k}\left(x_k^{(\xi_k)} \right)^2 \ ,
	\end{split}
\end{equation*}
since the equality \eqref{3.03}, when $\xi_r=e$ and $a\in\mathfrak{A}_e\setminus\{0\}$, only is possible if $\delta_{r}=\lambda_{rr}=0$. 

A special type of graded polynomial of degree $2$ is given by a graded polynomial $g\in \mathbb{F}\langle X^{\mathsf{G}}\rangle$ whose monomials are $\mathsf{G}$-homogeneous of degree $e$, i.e. all the monomials of $g$ belong to $(\mathbb{F}\langle X^{\mathsf{G}}\rangle)_e$ (see Definition \ref{3.26}). With this in mind, we have the following result:

\begin{corollary}\label{3.22}
Let $\mathsf{G}$ be a group, $\mathbb{F}$ a field with $|\mathbb{F}|>2$, $\mathfrak{A}$ an $\mathbb{F}$-algebra with a $\mathsf{G}$-grading $\Gamma$. Suppose that $\mathfrak{A}$ satisfies a graded polynomial identity $g=g(x^{(\xi_1)}_1, \dots, x^{(\xi_n)}_n)\in\mathbb{F}\langle X^{\mathsf{G}}\rangle$ of degree $2$. If $g$ is homogeneous of degree $e$, then $\mathfrak{A}$ satisfies a graded polynomial identity of the form 
\begin{equation*}
	\begin{split}
\widehat{g}(x^{(\xi_1)}_1, \dots, x^{(\xi_n)}_n;y^{(\xi_1^{-1})}_1, \dots, y^{(\xi_n^{-1})}_n)=\sum_{1\leq r \leq n} \lambda_{r} \left[x^{(\xi_r)}_r, y^{(\xi_r^{-1})}_r \right]_{\mathsf{f_{G}}}
+\sum_{\substack{1\leq k \leq n \\ \mathsf{o}(\xi_k)=2}} \delta_{k}\left(x_k^{(\xi_k)} \right)^2 \ ,
	\end{split}
\end{equation*}
where $\lambda_{r}, \delta_k\in\mathbb{F}$ and $[\ , \ ]_{\mathsf{f_G}}$ is an $\mathsf{f_{G}}$-commutator of $\mathfrak{A}$ which depends of $\Gamma$. In addition, if $\mathsf{char}(\mathbb{F})\neq2$, $\mathfrak{A}$ is not nilpotent and $\mathsf{G}$ has odd order, then $\widehat{g}$ can be rewritten as 
\begin{equation*}
\widehat{g}(x^{(\xi_1)}_1, \dots, x^{(\xi_n)}_n;y^{(\xi_1^{-1})}_1, \dots, y^{(\xi_n^{-1})}_n)=\sum_{r=1}^n \lambda_{r} \left[x^{(\xi_r)}_r, y^{(\xi_r^{-1})}_r \right]_{\mathsf{f_{G}}} \ .
\end{equation*}
\end{corollary}
\begin{proof}
By Lemma \ref{3.13}, we can assume that $g$ is a polynomial as in \eqref{3.02}. Since the $\mathsf{G}$-homogeneous degree of each monomial of $g$ is $e$, i.e. $x^{(\xi_r)}_r x^{(\xi_s)}_s, x^{(\xi_s)}_s x^{(\xi_r)}_r, (x_k^{(\xi_k)} )^2\in (\mathbb{F}\langle X^{\mathsf{G}}\rangle)_e$, we must have $\xi_r\xi_s=\xi_s\xi_r=\xi_k^2=e$ for all $r,s,k$. From this, $\xi_s=\xi_r^{-1}$ and $\xi_k^{-1}=\xi_k$ (i.e. $\mathsf{o}(\xi_k)=2$).

To conclude, as the order of $\mathsf{G}$ is divisible by the orders of all its elements, if $|\mathsf{G}|$ is odd, then $\mathsf{G}$ have not elements of order $2$, and hence, $\delta_k=0$ when $\xi_k\neq e$. The conclusion follows from Lemma \ref{mardua4.7}.
\end{proof}

%
%
\subsection{The graded identities of degree 2 which $M_k(\mathbb{F}^{\sigma}[\mathsf{H}])$ can satisfy}
Let $\mathfrak{A}$ be an $\mathbb{F}$-algebra with a $\mathsf{G}$-grading. Suppose that $\mathfrak{A}$ is $\mathsf{G}$-simple, has finite dimension and satisfies a graded polynomial identity $g\in\mathbb{F}\langle X^{\mathsf{G}}\rangle$ of degree $2$. Our main aim here is to deduce some properties of $g$. To this end, since Lemma \ref{Bahtteo3} ensures, for $\mathbb{F}$ an algebraically closed field of characteristic zero, that $\mathfrak{A}$ is $\mathsf{G}$-isomorphic to $\mathfrak{B}=M_{k}(\mathbb{F}^{\sigma}[\mathsf{H}])$, where $\mathfrak{B}$ is graded with an elementary-canonical $\mathsf{G}$-grading, we will study $g$ when $g\equiv_{\mathsf{G}}0$ in $M_{k}(\mathbb{F}^{\sigma}[\mathsf{H}])$.

In what follows, let us consider and study $\mathfrak{B}= M_k(\mathbb{F}^{\sigma}[\mathsf{H}])$, the $k\times k$ matrix algebra over the twisted group algebra $\mathbb{F}^{\sigma}[\mathsf{H}]$, with an elementary-canonical grading $\Gamma$ defined by a $k$-tuple $(\theta_1,\theta_2,\dots,\theta_k)\in\mathsf{G}^k$, where $\mathsf{H}$ is a subgroup of $\mathsf{G}$ and $\sigma$ is a $2$-cocycle on $\mathsf{H}$.

Relative to the order of $\mathsf{Supp}(\Gamma)$, it is easy to see that $|\mathsf{H}|\leq |\mathsf{Supp}(\Gamma)|\leq k^2|\mathsf{H}|$. Note that $k>1$ when $|\mathsf{Supp}(\Gamma)|>|\mathsf{H}|$. More precisely, if $m=\#\{\theta_1\mathsf{H},\theta_2\mathsf{H},\dots,\theta_k\mathsf{H}\}$ is the number of distinct left cosets determined by $\theta_1,\theta_2,\dots,\theta_k$, then $|\mathsf{Supp}(\Gamma)|\geq m|\mathsf{H}|$, since $E_{1i}\eta_\xi$ and $E_{1j}\eta_\zeta$ not belong to the same homogeneous component of $\mathfrak{B}$ when $\theta_i\notin \mathsf{H} \theta_j$. From this, if $\theta_i\notin \mathsf{H} \theta_j$ for all $i\neq j$, then $k|\mathsf{H}|\leq |\mathsf{Supp}(\Gamma)|\leq k^2|\mathsf{H}|$. On the other hand, if $\theta_1,\dots,\theta_k$ belong to normalizer of $\mathsf{H}$ in $\mathsf{G}$, we have that $\{\theta_i^{-1}\zeta\theta_i\in\mathsf{G}: i=1,\dots,k, \zeta\in\mathsf{H}\}=\bigcup_{i=1}^k \mathsf{H}^{\theta_i}\subseteq\mathsf{H}$, and hence, $|\mathsf{Supp}(\Gamma)|\leq (k^2-k+1)|\mathsf{H}|$. It is important to note that $\theta_i\notin \mathsf{H} \theta_j$ if and only if $\theta_i^{-1}\notin \theta_j^{-1}\mathsf{H}$. Recall that $\mathsf{H}^{\xi}=\{\zeta^\xi=\xi^{-1}\zeta\xi\in\mathsf{G}: \zeta\in\mathsf{H}\}$.

Now, consider the twisted group algebra $\mathfrak{D}=\mathbb{F}^{\sigma}[\mathsf{H}]$, for some subgroup $\mathsf{H}$ of $\mathsf{G}$ and $2$-cocycle $\sigma: \mathsf{H}\times\mathsf{H}\rightarrow \mathbb{F}^*$. Since $(\lambda\eta_\xi)(\gamma\eta_\zeta)=\lambda\gamma\sigma(\xi,\zeta)\eta_{\xi\zeta}$ for any $\xi,\zeta\in\mathsf{H}$ and $\lambda,\gamma\in\mathbb{F}$, putting $\mathsf{f}(\xi,\zeta)=\sigma(\xi,\zeta)^{-1}$ for any $\xi,\zeta\in\mathsf{H}$, we have that $[x^{(\xi)},y^{(\zeta)}]_{\mathsf{f_G}}$ is an $\mathsf{f_G}$-commutator of $\mathfrak{D}$ which depends of $\mathsf{H}$. Observe that $[\lambda\eta_\xi,\gamma\eta_\zeta]_{\mathsf{f_G}}=\lambda\gamma(\mathsf{f}(\xi,\zeta)\sigma(\xi,\zeta)\eta_{\xi\zeta}-\mathsf{f}(\zeta,\xi)\sigma(\zeta,\xi)\eta_{\zeta\xi})=\lambda\gamma(\eta_{\xi\zeta}-\eta_{\zeta\xi})$, and hence, $\mathfrak{D}$ is $\mathsf{f_G}$-commutative when $\mathsf{H}$ is abelian. Let us use this reasoning to define the $\sigma$-commutator of $\mathfrak{B}=M_k(\mathbb{F}^{\sigma}[\mathsf{H}])$. We define \textit{the $\sigma$-commutator of $\mathfrak{B}$} by the map that linearly extends the following application: 
\begin{equation*}
[E_{ij}\eta_\xi,E_{rs}\eta_\zeta]_{\sigma}=\frac{1}{\sigma(\xi,\zeta)} (E_{ij}\eta_\xi)( E_{rs}\eta_\zeta) -\frac{1}{\sigma(\zeta,\xi)} (E_{rs}\eta_\zeta) (E_{ij}\eta_\xi) \ ,
\end{equation*}
for any $\xi,\zeta\in\mathsf{H}$ and $i,j,r,s=1,2,\dots,k$.

\begin{lemma}\label{3.23}
Let $\mathfrak{B}=M_k(\mathbb{F}^{\sigma}[\mathsf{H}])$ with an elementary-canonical $\mathsf{G}$-grading $\Gamma$ defined by a $k$-tuple $(\theta_1,\dots,\theta_k)\in \mathsf{G}^k$. The following statements are true. 
	\begin{enumerate}
	\item[i)] If $[E_{ij}\eta_\xi,E_{rs}\eta_\zeta]_{\sigma}\equiv_{\mathsf{G}}0$ in $\mathfrak{B}$, then either $i\neq s$ and $j\neq r$ or $i=j=s=r$ and $\xi\zeta=\zeta\xi$;
	\end{enumerate}
Now, suppose $\theta_r\notin \mathsf{H} \theta_s $ for all $r\neq s$. 
	\begin{enumerate}
	\item[ii)] If $\mathsf{H}$ is normal in $\mathsf{G}$, then $k|\mathsf{H}|\leq |\mathsf{Supp}(\Gamma)|\leq (k^2-k+1)|\mathsf{H}|$;
	\item[iii)] If $[x^{(\xi)},y^{(\xi^{-1})}]_{\sigma}\equiv_{\mathsf{G}}0$ in $\mathfrak{B}$, for some $\xi\in\mathsf{G}$, then either $\xi\notin\mathsf{Supp}(\Gamma)$ or $\xi\in\mathsf{H}^{\theta_r}$ for some $r\in\{1,\dots,k\}$;
	\item[iv)] For any $\xi,\zeta\in \mathsf{H}$, if $\xi\zeta=\zeta\xi$, then $[x^{(\xi^{\theta_r})},y^{(\zeta^{\theta_s})}]_{\sigma}\equiv_{\mathsf{G}}0$ in $\mathfrak{B}$ for all $r,s=1,2,\dots,n$;
	\item[v)] Suppose $|\mathsf{Supp}(\Gamma)|=k|\mathsf{H}|$. If $[x^{(\xi)},y^{(\zeta)}]_{\sigma}\equiv_{\mathsf{G}}0$ in $\mathfrak{B}$ for some $\xi,\zeta\in \mathsf{G}$, then either $\{\xi,\zeta\}\not\subset \mathsf{Supp(\Gamma)}$ or $\xi,\zeta\in \mathsf{H}^{\theta_r}$ for all $r\in\{1,\dots,k\}$.
	\end{enumerate}
\end{lemma}
\begin{proof}
The items \textit{i)} and \textit{ii)} follow from the observations above;

\textit{iii)} Given any $\xi\in\mathsf{Supp}$, take $E_{ij}\eta_\zeta\in\mathfrak{B}_\xi$, and so $E_{ji}\eta_{\zeta^{-1}}\in\mathfrak{B}_{\xi^{-1}}$. From this, $[E_{ij}\eta_\zeta,E_{ji}\eta_{\zeta^{-1}}]_{\sigma}=0$, and hence, by item \textit{i)}, it follows that $i=j$. Therefore, we have that $\xi=\theta_i^{-1}\zeta\theta_i \in\mathsf{H}^{\theta_i}$.

\textit{iv)} Suppose $\xi,\zeta\in \mathsf{H}$ such that $\xi\zeta=\zeta\xi$. Fixed any $r,s\in\{1,2,\dots,k\}$, we have that $E_{rr}\eta_\xi\in\mathfrak{B}_{\xi^{\theta_r}}$ and $E_{ss}\eta_\zeta\in\mathfrak{B}_{\zeta^{\theta_s}}$. By item \textit{i)}, if $r\neq s$, it follows that $[x^{(\xi^{\theta_r})},y^{(\zeta^{\theta_s})}]_{\sigma}\equiv_{\mathsf{G}}0$ in $\mathfrak{B}$. Conversely, if $r=s$, then we have that
\begin{equation*}
	\begin{split}
[E_{rr}\eta_\xi,E_{rr}\eta_\zeta]_{\sigma} &= \frac{1}{\sigma(\xi,\zeta)} (E_{rr}\eta_\xi)( E_{rr}\eta_\zeta) -\frac{1}{\sigma(\zeta,\xi)} (E_{rr}\eta_\zeta) (E_{rr}\eta_\xi) \\
	&= \frac{1}{\sigma(\xi,\zeta)} \sigma(\xi,\zeta) E_{rr}\eta_{\xi\zeta}-\frac{1}{\sigma(\zeta,\xi)}\sigma(\zeta,\xi) E_{rr}\eta_{\zeta \xi} \\
	&=E_{rr}\eta_{\xi\zeta}- E_{rr}\eta_{\zeta \xi} =0 \ ,
	\end{split}
\end{equation*}
because $\xi\zeta=\zeta \xi$, and so $\eta_{\xi\zeta}=\eta_{\zeta \xi}$.

\textit{v)} Assume $|\mathsf{Supp}(\Gamma)|=k|\mathsf{H}|$. Since $\theta_i\notin \mathsf{H} \theta_j $ for all $i\neq j$, it is easy to see that, for each $i_0=1,\dots,k$, $\mathsf{Supp}(\Gamma)=\{\theta_{i_0}^{-1}\zeta\theta_j\in\mathsf{G}: \zeta\in\mathsf{H}, j=1,\dots,k\}$. Note that $E_{ij_1}\eta_{\zeta_1},E_{ij_2}\eta_{\zeta_2}\in\mathfrak{B}_\xi$ implies $j_1=j_2$ and $\zeta_1=\zeta_2$. Hence, for each pair $(i,\xi)\in \{1,\dots,k\}\times\mathsf{Supp}(\Gamma)$, there exists only pair $(j,\zeta)\in\{1,\dots,k\}\times \mathsf{H}$ such that $E_{ij}\eta_\zeta\in\mathfrak{B}_\xi$. Consequently, for each $\xi\in\mathsf{Supp}(\Gamma)$, there are $\zeta_{1\xi},\dots,\zeta_{k\xi} \in\mathsf{H}$ such that $\mathfrak{B}_\xi=\{E_{1j_1}\eta_{\zeta_{1\xi}},\dots,E_{kj_k}\eta_{\zeta_{k\xi}}\}$, where $j_1,\dots,j_k \in\{1,\dots,k\}$ are pairwise distinct.

Take $\xi,\zeta\in \mathsf{Supp}(\Gamma)$ such that $[x^{(\xi)},y^{(\zeta)}]_{\sigma}\equiv_{\mathsf{G}}0$ in $\mathfrak{B}$. \textbf{Claim:} $\xi,\zeta\in \mathsf{H}^{\theta_r}$ for all $r\in\{1,\dots,k\}$. In fact, by the first part of the proof of this item, for each $i_0=1,\dots,k$, we can take $E_{i_0 j}\eta_{\hat{\xi}}\in\mathfrak{B}_\xi$ and $E_{j s}\eta_{\hat{\zeta}}\in\mathfrak{B}_\zeta$. By hypotheses, it follows that $[E_{i_0 j}\eta_{\hat{\xi}},E_{j s}\eta_{\hat{\zeta}}]_{\sigma}=0$, and so, by item \textit{i)} of this proposition, we must have $j=s=i_0$. From this, $\xi=\theta_{i_0}^{-1}\hat{\xi}\theta_{i_0}$ and $\zeta=\theta_{i_0}^{-1}\hat{\zeta}\theta_{i_0}$, and so the result follows.
\end{proof}

The Example \ref{3.18} exhibits a matrix algebra (with an elementary grading) which does not satisfy the item \textit{iv)} of the previous result. Already the matrix algebra of Example \ref{3.16} does not satisfy the item \textit{v)} because the hypotheses ``$|\mathsf{Supp}(\Gamma)|=k|\mathsf{H}|$'' is not satisfied.


In Group Theory, the Lagrange's Theorem ensures that $|\mathsf{G}|=|\mathsf{H}|\cdot(\mathsf{G}:\mathsf{H})$ for any subgroup $\mathsf{H}$ of a finite group $\mathsf{G}$, where $(\mathsf{G}:\mathsf{H})$ is the index of $\mathsf{H}$ in $\mathsf{G}$. Hence, given a subgroup $\mathsf{H}$ of $\mathsf{G}$ and a $k$-tuple $(\theta_1,\dots,\theta_k)\in \mathsf{G}^k$, if $k=(\mathsf{G}:\mathsf{H})$ and $\theta_r\notin \mathsf{H} \theta_s $ for all $r\neq s$, then $\mathsf{G}=\{\theta_{r_0}^{-1}\xi\theta_s: \xi\in \mathsf{H}, s=1,\dots,k\}$, for all $r_0=1,\dots,k$. With this in mind, the following result is immediate from Lemma \ref{3.23} (and its proof).

\begin{corollary}\label{3.24}
Let $\mathfrak{B}=M_k(\mathbb{F}^{\sigma}[\mathsf{H}])$ with an elementary-canonical grading $\Gamma$ defined by a $k$-tuple $(\theta_1,\theta_2,\dots,\theta_k)\in \mathsf{G}^k$. Suppose $k=(\mathsf{G}:\mathsf{H})$ and $\theta_r\notin \mathsf{H} \theta_s $ for all $r\neq s$.
	\begin{enumerate}
	\item[i)] $(x^{(\xi)})^2\not\equiv_{\mathsf{G}}0$ and $y^{(\zeta)}z^{(\varsigma)}\not\equiv_{\mathsf{G}}0$ in $\mathfrak{B}$ for any $\xi,\zeta,\varsigma\in\mathsf{G}$;
	\item[ii)] If $\xi\notin \mathsf{H}^{\theta_i}$ for some $i\in\{1,\dots,k\}$, then $[x^{(\xi)},y^{(\zeta)}]_{\sigma}\not\equiv_{\mathsf{G}}0$ in $\mathfrak{B}$ for any $\zeta\in\mathsf{G}$;
	\item[iii)] If $[x^{(\xi)},y^{(\zeta)}]_{\sigma}\equiv_{\mathsf{G}}0$ in $\mathfrak{B}$ for any $\xi,\zeta\in \mathsf{H}$, then $\mathsf{H}$ is abelian and a normal subgroup of $\mathsf{G}$;
	\item[iv)] If $k>1$, then there is $\xi\in\mathsf{G}$ such that $[x^{(\xi)},y^{(\zeta)}]_{\sigma}\not\equiv_{\mathsf{G}}0$ in $\mathfrak{B}$ for any $\zeta\in\mathsf{G}$.
	\end{enumerate}
\end{corollary}

We now apply Lemma \ref{3.13}, Lemma \ref{3.23}, Corollary \ref{3.24} and above observations to describe a graded polynomial identity of degree $2$ which the matrix algebra $M_k(\mathbb{F}^{\sigma}[\mathsf{H}])$ (with an elementary-canonical $\mathsf{G}$-grading) can satisfy.

\begin{proposition}\label{3.19}
Let $\mathbb{F}$ be a field with $|\mathbb{F}|>2$, $\mathsf{G}$ a finite group, $\mathsf{H}$ a subgroup of $\mathsf{G}$, and $\sigma$ a $2$-cocycle on $\mathsf{H}$. Let $\mathfrak{B}=M_k(\mathbb{F}^\sigma[\mathsf{H}])$ the algebra of $k\times k$ matrices over $\mathbb{F}^\sigma[\mathsf{H}]$ with an elementary-canonical grading $\Gamma$ defined by a $k$-tuple $(\theta_1,\theta_2,\dots,\theta_k)\in \mathsf{G}^k$. Suppose $k=(\mathsf{G}:\mathsf{H})$ and $\theta_r\notin \mathsf{H} \theta_s $ for all $r\neq s$. If $\mathfrak{B}$ satisfies a graded polynomial identity $g\in\mathbb{F}\langle X^{\mathsf{G}} \rangle$ of degree $2$, then  $g$ can be rewritten as 
\begin{equation*}
	\begin{split}
g=\sum_{\substack{i,j=1 \\ \xi,\zeta\in\mathsf{H}}}^k \lambda_{i,j}^{\xi,\zeta} \left[x^{(\xi^{\theta_i})}, y^{(\zeta^{\theta_j})} \right]_{\sigma} 
\ ,
	\end{split}
\end{equation*}
for some $\lambda_{i,j}^{\xi,\zeta}\in\mathbb{F}$.
\end{proposition}

Observe that, since $[x^{(\xi^{\theta_i})}, y^{(\zeta^{\theta_j})} ]_{\sigma}=-[ y^{(\zeta^{\theta_j})}, x^{(\xi^{\theta_i})} ]_{\sigma}$, the previous result can be improved by
\begin{equation*}
	\begin{split}
g=\sum_{\substack{ 1\leq i,j\leq k \\ 1\leq r \leq s \leq m }} \lambda_{r,s}^{i,j} \left[x^{(\xi_r^{\theta_i})}, y^{(\xi_s^{\theta_j})} \right]_{\sigma} 
\ , 
	\end{split}
\end{equation*}
where $\mathsf{H}=\{\xi_1, \xi_2,\dots,\xi_m\}$ and $ \lambda_{r,s}^{i,j}\in\mathbb{F}$.

Lastly, given $\mathsf{G}$ a group and $\mathbb{F}$ a field, and considering the description presented in Lemma \ref{3.13}, another special type of graded polynomial of degree $2$ of $\mathbb{F}\langle X^{\mathsf{G}}\rangle$ is the following:
	\begin{equation}\label{3.15}
g(x^{(e)}_1, \dots, x^{(e)}_n;y^{(\xi_1)}_1, \dots, y^{(\xi_m)}_m)=\sum_{r=1}^m \gamma_{r} \left[x^{(e)}_r, y^{(\xi_r)}_r \right]_{\mathsf{f_G}} +\sum_{s=1}^n \delta_{s}\left(x_s^{(e)} \right)^2 \ ,
	\end{equation}
where $n\geq m$, $\gamma_{1},\dots,\gamma_m, \delta_{1},\dots,\delta_n\in\mathbb{F}$, and $[\ , \ ]_{\mathsf{f_G}}$ is an $\mathsf{f_{G}}$-commutator. In the next result, let us consider graded polynomials of this type.

\begin{lemma}\label{3.04}
	Let $\mathsf{H}$ be a subgroup of a group $\mathsf{G}$, $\mathbb{F}$ a field, and $\mathfrak{B}=M_k(\mathbb{F}^\sigma[\mathsf{H}])$ the algebra of $k\times k$ matrices over $\mathbb{F}^\sigma[\mathsf{H}]$ with an elementary-canonical $\mathsf{G}$-grading $\Gamma$ defined by a $k$-tuple $(\theta_1,\dots,\theta_k)\in \mathsf{G}^k$, where $\sigma$ is a $2$-cocycle on $\mathsf{H}$. Consider any map $\mathsf{f_G}:\mathsf{G}\times\mathsf{G}\rightarrow\mathbb{F}$. The following statements are true:
	\begin{enumerate}
	\item[i)] Given $\xi\in\mathsf{Supp}(\Gamma)$, if the graded polynomial $\left[x^{(e)}, y^{(\xi)} \right]_{\mathsf{f_G}}$ is a nontrivial graded identity for $\mathfrak{B}$, then $E_{ij}\eta_\zeta\in\mathfrak{B}_\xi$ only if $i=j$, and so $\xi\in\mathsf{H}^{\theta_i}$. In addition, $\left[x^{(e)}, y^{(\xi)} \right]\equiv_{\mathsf{G}}0$ in $\mathfrak{B}$;
	
	\item[ii)] If for each $\xi\in\mathsf{Supp}(\Gamma)$ there is $\mathsf{f_G}:\mathsf{G}\times\mathsf{G}\rightarrow\mathbb{F}$ such that $\left[x^{(e)}, y^{(\xi)} \right]_{\mathsf{f_G}}\neq 0$ and $ \left[x^{(e)}, y^{(\xi)} \right]_{\mathsf{f_G}}\equiv_{\mathsf{G}} 0$ in $\mathfrak{B}$, then $\mathfrak{B}=\mathbb{F}^\sigma[\mathsf{H}]$. In addition, $\mathfrak{B}_e$ is central in $\mathfrak{B}$;
	
	\item[iii)] If $ \left[x^{(e)}, y^{(e)} \right]_{\mathsf{f_G}}$ is a nontrivial graded polynomial identity for $\mathfrak{B}$, then $\mathfrak{B}_e$ is commutative and $\theta_i\notin\mathsf{H}\theta_j$ for $i\neq j$. In this case, $\mathfrak{B}_e$ is central in $\mathfrak{B}_{D}= \mathsf{span}_{\mathbb{F}}\{E_{ii}\eta_\zeta\in\mathfrak{B}: \zeta\in\mathsf{H}, i=1,\dots, k\}$.
	\end{enumerate}
Moreover, if $g=g(x^{(e)}_1, \dots, x^{(e)}_n;y^{(\xi_1)}_1, \dots, y^{(\xi_m)}_m)\in\mathbb{F}\langle X^\mathsf{G} \rangle$ is a graded polynomial of degree $2$ as in \eqref{3.15} and $g\equiv_{\mathsf{G}}0$ in $\mathfrak{B}$, then $m=n$ and $\displaystyle g(x^{(e)}_1, \dots, x^{(e)}_m;y^{(\xi_1)}_1, \dots, y^{(\xi_m)}_m)=\sum_{r=1}^m \tilde{\gamma}_{r} \left[x^{(e)}_r, y^{(\xi_r)}_r \right]$, for $\tilde{\gamma}_1,\dots,\tilde{\gamma}_m\in\mathbb{F}$, not all zero.
\end{lemma}
\begin{proof}
\textit{i)} Suppose $\left[x^{(e)}, y^{(\xi)} \right]_{\mathsf{f_G}}\equiv_{\mathsf{G}} 0$ in $\mathfrak{B}$, for some $\xi\in\mathsf{Supp}(\Gamma)$, with $[x^{(e)}, y^{(\xi)} ]_{\mathsf{f_G}}$ nontrivial. Hence, either $\mathsf{f}(e,\xi)\neq0$ or $\mathsf{f}(\xi,e)\neq0$, and taking $i,j\in\{1,\dots,k\}$ and $\zeta\in\mathsf{H}$ such that $E_{ij}\eta_\zeta\in\mathfrak{B}_\xi$, we have that $[E_{ll}\eta_e, E_{ij}\eta_{\zeta}]_{\mathsf{f_G}}=0$ for all $l\in\{1,\dots,k\}$, since $E_{ll}\eta_e\in\mathfrak{B}_e$. Assume $\mathsf{f}(e,\xi)\neq0$. So $0=[E_{ii}\eta_e, E_{ij}\eta_{\zeta}]_{\mathsf{f_G}}=\mathsf{f}(e,\xi)\sigma(e,\zeta)E_{ij}\eta_{\zeta}-\mathsf{f}(\xi,e)(E_{ij}\eta_{\zeta} E_{ii}\eta_{e})$. As $\mathsf{f}(e,\xi)\sigma(e,\zeta)E_{ij}\eta_{\zeta}\neq0$, we must conclude that $\mathsf{f}(\xi,e)(E_{ij}\eta_{\zeta} E_{ii}\eta_{e})\neq0$, and so $i=j$. From this, it follows that $\mathsf{f}(\xi,e)=\mathsf{f}(e,\xi)$ and $E_{ij}\eta_\zeta\in\mathfrak{B}_\xi$ implies $i=j$ and $\xi\in \mathsf{H}^{\theta_i}$. Note that $[x^{(e)}, y^{(\xi)}]=(\mathsf{f}(e,\xi))^{-1}[x^{(e)}, y^{(\xi)}]_{\mathsf{f_G}}\equiv_{\mathsf{G}} 0$ in $\mathfrak{B}$.

\textit{ii)} To obtain a contradiction, assume $k>1$. Since $0\neq E_{1k}\eta_e\in\mathfrak{B}_{\theta_1^{-1}\theta_k}$, we have that $\theta_1^{-1}\theta_k\in\mathsf{Supp}(\Gamma)$, and hence, by the hypothesis, there is $\mathsf{f_G}:\mathsf{G}\times\mathsf{G}\rightarrow\mathbb{F}$ such that $[x^{(e)}, y^{(\theta_1^{-1}\theta_k)} ]_{\mathsf{f_G}}$ is a nontrivial graded polynomial identity for $\mathfrak{B}$. As $E_{11}\eta_e,E_{kk}\eta_e\in\mathfrak{B}_{e}$, it follows that $0=[E_{11}\eta_e,E_{1k}\eta_e]_{\mathsf{f_G}}=\mathsf{f}(e,\theta_1^{-1}\theta_k)\sigma(e,e)E_{1k}\eta_e$ and 
$0=[E_{kk}\eta_e,E_{1k}\eta_e]_{\mathsf{f_G}}=-\mathsf{f}(\theta_1^{-1}\theta_k,e)\sigma(e,e)E_{1k}\eta_e$, and thus, $\mathsf{f}(e,\theta_1^{-1}\theta_k)=\mathsf{f}(\theta_1^{-1}\theta_k,e)=0$. This contradicts the fact that $[x^{(e)}, y^{(\theta_1^{-1}\theta_k)} ]_{\mathsf{f_G}}$ is nontrivial. Therefore, $\mathfrak{B}=\mathbb{F}^\sigma[\mathsf{H}]$.

Now, it is well-known that $\sigma(e,\zeta)=\sigma(\zeta,e)$, for any $\zeta\in\mathsf{H}$ (for a proof of this fact, see \cite{MarduaThesis}, Proposition 1.2.6 , p.26). So $[\lambda \eta_e, \gamma \eta_\zeta]=\lambda \gamma \eta_e \eta_\zeta-\gamma\lambda \eta_\zeta \eta_e=\lambda \gamma \sigma(e,\zeta) \eta_\zeta-\gamma\lambda \sigma(\zeta,e)\eta_\zeta=0$ for any $\lambda,\gamma\in\mathbb{F}$ and $\zeta\in\mathsf{G}$. The item \textit{ii)} follows.

\textit{iii)} Suppose $[x^{(e)}, y^{(e)} ]_{\mathsf{f_G}}\neq 0$ and $ [x^{(e)}, y^{(e)}]_{\mathsf{f_G}}\equiv_{\mathsf{G}} 0$ in $\mathfrak{B}$. By item \textit{i)}, $E_{ij}\eta_\zeta\in\mathfrak{B}_e$ iff $i=j$ and $\zeta=e$. So $\mathfrak{B}_e=\mathsf{span}_\mathbb{F}\{E_{ii}\eta_e: i=1,\dots,k\}$, and obviously $\mathfrak{B}_e$ is commutative (and also central in $\mathfrak{B}_{D}$). Now, let $r,s\in\{1,\dots,k\}$ such that $\theta_r\in\mathsf{H}\theta_s$. Hence, taking $\varsigma\in\mathsf{H}$ such that $\theta_r=\varsigma\theta_s$, it follows that $E_{rs}\eta_\varsigma\in\mathfrak{B}_e$. Again by the item \textit{i)}, we have that $r=s$ and $\varsigma=e$. The item \textit{iii)} is proved.

Finally, consider  $g(x^{(e)}_1, \dots, x^{(e)}_n;y^{(\xi_1)}_1, \dots, y^{(\xi_m)}_m)=\sum_{r=1}^m \gamma_{r} [x^{(e)}_r, y^{(\xi_r)}_r ]_{\mathsf{f_G}} +\sum_{s=1}^n \delta_{s}(x_s^{(e)} )^2 $ as in \eqref{3.15}. Since $E_{ii}\eta_e\in\mathfrak{B}_e$ and $\left(E_{ii}\eta_e\right)^2=\sigma(e,e)E_{ii}\eta_e\neq0$ for all $i=1,\dots,k$, it follows that $\delta_{s}=0$ for $s=1,\dots,n$, and so $m=n$. Consequently, $ g(x^{(e)}_1, \dots, x^{(e)}_m;y^{(\xi_1)}_1, \dots, y^{(\xi_m)}_m)=\sum_{r=1}^m \gamma_{r} [x^{(e)}_r, y^{(\xi_r)}_r ]_{\mathsf{f_G}}$. Now, by item \textit{i)} of this lemma, we can write $\xi_r=\theta_{r_i}^{-1}\zeta\theta_{r_i}$, with $\zeta\in\mathsf{H}$ and $r_i\in\{1,\dots,k\}$. Hence, $0=[E_{r_i r_i}\eta_e, E_{r_i r_i}\eta_\zeta ]_{\mathsf{f_G}}=\mathsf{f}(e,\xi_r)\sigma(e,\zeta) E_{r_i r_i}\eta_\zeta-\mathsf{f}(\xi_r,e)\sigma(\zeta,e) E_{r_i r_i}\eta_\zeta =(\mathsf{f}(e,\xi_r)-\mathsf{f}(\xi_r,e))\sigma(e,\zeta) E_{r_i r_i}\eta_\zeta$. So $\mathsf{f}(e,\xi_r)=\mathsf{f}(\xi_r,e)$ for all $r_i\in\{1,\dots,k\}$, and thus, $[x^{(e)}_r, y^{(\xi_r)}_r ]_{\mathsf{f_G}}=\mathsf{f}(e,\xi_r)x^{(e)}_r y^{(\xi_r)}_r-\mathsf{f}(\xi_r,e) y^{(\xi_r)}_r x^{(e)}_r=\mathsf{f}(e,\xi_r)(x^{(e)}_r y^{(\xi_r)}_r - y^{(\xi_r)}_r x^{(e)}_r)=\mathsf{f}(e,\xi_r)[x^{(e)}_r , y^{(\xi_r)}_r ]$. Therefore, putting $\tilde{\gamma}_r=\gamma_r\mathsf{f}(e,\xi_r)\in\mathbb{F}$ for all $r=1,\dots,m$, we conclude that $g(x^{(e)}_1, \dots, x^{(e)}_m;y^{(\xi_1)}_1, \dots, y^{(\xi_m)}_m)=\sum_{r=1}^m \tilde{\gamma}_{r} [x^{(e)}_r, y^{(\xi_r)}_r ]$.
\end{proof}

%
%
%

\section{Main Results}\label{main}

In this section, the main aim is to present a study on the graded algebras that satisfy a graded polynomial identity $g=g(x_1^{(e)}, \dots, x_n^{(e)})$ of degree $2$. Particularly, let us study the following problem:

\begin{problem}\label{3.01}
What can we say about an associative algebra $\mathfrak{A}$ graded by a group $\mathsf{G}$ when its neutral component $\mathfrak{A}_e$ satisfies a polynomial identity $g$ of degree $2$? Which ordinary identities does $\mathfrak{A}$ satisfy?
\end{problem}

More general than that, our main aim in this section is to deduce some properties of $\mathfrak{A}$ when $\mathfrak{A}$ satisfies some graded polynomial $g$ of type \eqref{3.15}.

Let $\mathsf{G}$ be a group and $\mathbb{F}$ a field. Let us initially consider a $\mathsf{G}$-graded $\mathbb{F}$-algebra $\mathfrak{A}$ that satisfies a graded polynomial identity $g$ of degree $2$ of $\mathbb{F}\langle X^{\mathsf{G}}\rangle$. We must have one of the two possibilities below:
\begin{enumerate}
	\item[1)] $\mathfrak{A}$ is nilpotent;
	\item[2)] $\mathfrak{A}$ is not nilpotent.
\end{enumerate}

Suppose $\mathsf{G}$ a finite group. By Lemma \ref{mardua3.9}, $\mathfrak{A}$ is nilpotent when $\mathfrak{A}_e$ is nilpotent. From this, studying the above possibilities is equivalent to studying ``$\mathfrak{A}_e$ is nilpotent'' or ``$\mathfrak{A}_e$ is not nilpotent''. On the other hand, Lemma \ref{mardua4.7} ensures that if $\mathfrak{A}_e$ is not nilpotent and $\mathsf{char}(\mathbb{F})\neq2$, then $\mathfrak{A}_e$ is not nil of index $2$. So, if $g$ is a polynomial of type \eqref{3.15}, $\mathsf{char}(\mathbb{F})\neq2$ and $\mathfrak{A}$ is not nilpotent, we must have that all $\delta_s$'s of $g$ are zero, and hence, $g(x^{(e)}_1, \dots, x^{(e)}_m;y^{(\xi_1)}_1, \dots, y^{(\xi_m)}_m)=\sum_{r=1}^m \gamma_{r} [x^{(e)}_r, y^{(\xi_r)}_r ]_{\mathsf{f_G}}$, where $\gamma_{1},\dots,\gamma_m\in\mathbb{F}$ and $[\ , \ ]_{\mathsf{f_G}}$ is an $\mathsf{f_{G}}$-commutator.

	Now, consider $\mathfrak{A}$ a finite dimensional $\mathsf{G}$-graded $\mathbb{F}$-algebra, with $\mathfrak{A}=\left(\bigtimes_{i=1}^p M_{k_i}(\mathbb{F}^{\sigma_i}[\mathsf{H}_i])\right) \oplus \mathsf{J}(\mathfrak{A})$ as described by Lemma \ref{teoIrina03}. Recall that $\mathsf{J}(\mathfrak{A})$ is the largest nilpotent ideal of $\mathfrak{A}$ (see Chapter 1 of \cite{Hers05}), and so $\mathfrak{A}\neq\mathsf{J}(\mathfrak{A})$ when $\mathfrak{A}$ is not nilpotent. In this way, if $\mathfrak{A}$ is not nilpotent and satisfies a graded polynomial identity of degree $2$ of the type \eqref{3.15}, then it follows from Lemma \ref{3.04} that $\mathfrak{A}$ satisfies a graded polynomial identity $\widetilde{g}(x^{(e)}_1, \dots, x^{(e)}_m;y^{(\xi_1)}_1, \dots, y^{(\xi_m)}_m)=\sum_{r=1}^m \gamma_{r} [x^{(e)}_r, y^{(\xi_r)}_r ]$, $\gamma_{r}\in\mathbb{F}^*$ and $\xi_r\in\mathsf{Supp}(\Gamma)$. 
%

By applying Lemmas \ref{teoBergCoheBaht}, \ref{teoIrina01} and \ref{teoIrina03}, along with the observations discussed above, it is evident that the following two results hold true.

	\begin{theorem}\label{3.20}
Let $\mathsf{G}$ be a finite abelian group, $\mathbb{F}$ an algebraically closed field of characteristic zero, and $\mathfrak{A}$ a finitely generated $\mathsf{G}$-graded algebra. 
Suppose that $\mathfrak{A}_e$ is a $PI$-algebra. If $\mathfrak{A}$ satisfies a graded polynomial identity $g\in \left(\mathbb{F}\langle X^\mathsf{G} \rangle\right)$ of degree $2$ of the type \eqref{3.15}, then $\mathfrak{A}$ is either a nilpotent algebra or satisfies a graded polynomial identity $\displaystyle \widetilde{g}(x^{(e)}_1, \dots, x^{(e)}_m;y^{(\xi_1)}_1, \dots, y^{(\xi_m)}_m)=\sum_{r=1}^m \tilde{\gamma}_{r} \left[x^{(e)}_r, y^{(\xi_r)}_r \right]$, where $\tilde{\gamma}_1,\dots,\tilde{\gamma}_m\in\mathbb{F}^*$ and $\xi_r\in\mathsf{Supp}(\Gamma)$. 
In addition, if $\mathsf{Supp}(\Gamma)=\{\xi_1, \dots, \xi_m\}$ and $\mathfrak{A}$ is not nilpotent, then $\mathfrak{A}/\mathsf{J}(\mathfrak{A})$ and $\mathbb{F}^{\sigma_1}[\mathsf{H}_1] \times \cdots \times \mathbb{F}^{\sigma_p}[\mathsf{H}_p]$ are $\mathsf{G}PI$-equivalent for some subgroups $\mathsf{H}_1,\dots,\mathsf{H}_p$ of $\mathsf{G}$ and $2$-cocycles $\sigma_1,\dots,\sigma_p$ on $\mathsf{H}_1,\dots,\mathsf{H}_p$, respectively.
	\end{theorem}

	\begin{corollary}\label{3.05}
Let $\mathsf{G}$ be a finite abelian group, $\mathbb{F}$ an algebraically closed field of characteristic zero, and $\mathfrak{A}$ a finitely generated $\mathsf{G}$-graded algebra. 
If $\mathfrak{A}_e$ satisfies a polynomial identity 
 of degree $2$, then either $\mathfrak{A}$ is a nilpotent algebra or $\mathfrak{A}_e$ is a commutative algebra.
	\end{corollary}

Note that Corollary \ref{3.05} is a particular case of Corollary \ref{3.22}.

%
%

\subsection{The $\mathsf{G}$-variety of graded algebras whose the neutral components satisfy a polynomial identity of degree 2}\label{subsecvariety}

Let $g=g(x_1^{(e)}, \dots, x_n^{(e)})\in\mathbb{F}\langle X^\mathsf{G} \rangle$ be a graded polynomial of degree $2$. Here, let us now consider and study the graded variety $\mathfrak{W}^\mathsf{G}=\mathsf{var}^\mathsf{G}(g)$ of $\mathsf{G}$-graded associative algebras which satisfy $g$. Recall that, given a subset $S\subset\mathbb{F}\langle X^\mathsf{G} \rangle$, the $\mathsf{G}$-variety $\mathsf{var}^\mathsf{G}(S)$ generated by $S$ is the class of graded associative algebras that satisfy all polynomials of $S$, i.e. a $\mathsf{G}$-graded algebra $\mathfrak{A}$ belongs to $\mathsf{var}^\mathsf{G}(S)$ iff $f\equiv_\mathsf{G}0$ in $\mathfrak{A}$ for any $f\in S$. 

Assume that $\mathbb{F}$ is an algebraically closed field of characteristic zero and $\mathsf{G}$ is a finite abelian group. By Lemmas \ref{teoIrina02} and \ref{teoIrina03}, there exists a $\mathsf{G}\times\mathbb{Z}_2$-graded finite dimensional algebra $\mathfrak{A} = \mathfrak{B}\oplus\mathsf{J}$ such that 
\begin{equation*}
	\begin{split}
\mathfrak{W}^\mathsf{G}=\mathsf{var}^\mathsf{G}(\mathsf{E}^\mathsf{G}(\mathfrak{A})) \ ,
	\end{split}
\end{equation*}
where $\mathsf{J}=\mathsf{J}(\mathfrak{A})$ is the Jacobson radical of $\mathfrak{A}$, and $\mathfrak{B}=M_{k_1}(\mathbb{F}^{\sigma_1}[\mathsf{H}_1]) \times \cdots \times M_{k_p}(\mathbb{F}^{\sigma_p}[\mathsf{H}_p])$, with $\mathsf{H}_i\leq \mathsf{G}\times\mathbb{Z}_2$, $\sigma_i\in\mathsf{Z}^2(\mathsf{H}_i, \mathbb{F}^*)$, and $M_{k_i}(\mathbb{F}^{\sigma_i}[\mathsf{H}_i])$ is graded with an elementary-canonical $(\mathsf{G}\times\mathbb{Z}_2)$-grading. 
As $g\equiv_{\mathsf{G}}0$ in $\mathsf{E}^\mathsf{G}(\mathfrak{A})$, we have that $g\equiv0$ in $\mathsf{E}^\mathsf{G}(\mathfrak{A})_e=\mathfrak{A}_{(e,0)}\otimes_\mathbb{F}\mathsf{E}_0 +\mathfrak{A}_{(e,1)}\otimes_\mathbb{F}\mathsf{E}_1$, and consequently $g\equiv0$ in $\mathfrak{A}_{(e,0)}\otimes_\mathbb{F}\mathsf{E}_0$, and in particular $g\equiv0$ in $\mathfrak{A}_{(e,0)}$ (because $\mathsf{E}_0$ is commutative). Hence, by Lemma \ref{mardua3.9}, if $\mathfrak{A}$ is not nilpotent, then we can conclude that $\mathfrak{B}_{(e,0)}\subseteq\mathfrak{A}_{(e,0)}$ is not zero, otherwise $\mathfrak{A}_{(e,0)}$ is equal to $\mathsf{J}_{(e,0)}$, which is nilpotent, and this leads us to a contradiction (see Lemma \ref{mardua3.9}). It means that $g\equiv_{\mathsf{G}}0$ in $\mathsf{E}^\mathsf{G}(\mathfrak{A})$ implies that either $\mathsf{E}^\mathsf{G}(\mathfrak{A})$ is nilpotent or $g\equiv0$ in $\mathfrak{B}_{(e,0)}\neq\{0\}$. Note that if $\mathfrak{B}_{(e,0)}\neq\{0\}$ and $g\equiv0$ in $\mathfrak{B}_{(e,0)}$, then $g\equiv0$ in $(M_{k_i}(\mathbb{F}^{\sigma_i}[\mathsf{H}_i]))_{(e,0)}$ for all $i=1,\dots,p$.

The following theorem follows from above observations, and Lemmas \ref{3.13} and \ref{3.04}. It is natural to conclude that it improves Corollary \ref{3.05}, since Lemmas \ref{teoIrina01} and \ref{teoIrina02} are connected.

\begin{theorem}\label{3.21}
Let $\mathsf{G}$ be a finite abelian group, $\mathbb{F}$ an algebraically closed field of characteristic zero and $g=g(x_1^{(e)}, \dots, x_n^{(e)})\in\mathbb{F}\langle X^\mathsf{G} \rangle$ a graded polynomial of degree $2$. The $\mathsf{G}$-variety $\mathfrak{W}^\mathsf{G}$ generated by $g$ is equal to either $\mathsf{var}^\mathsf{G}([x^{(e)},y^{(e)}])$ or $\mathsf{var}^\mathsf{G}(N)$ for some nilpotent $\mathsf{G}$-graded algebra $N$.
\end{theorem}

Recall that, given any $\mathsf{G}$-graded algebra $\mathfrak{A}$, $\mathsf{var}^\mathsf{G}(\mathfrak{A})\coloneqq \mathsf{var}^\mathsf{G}(\mathsf{T}^{\mathsf{G}}(\mathfrak{A}))$. Beside that, taking $f\in\mathbb{F}\langle X^\mathsf{G} \rangle$ any graded polynomial identity for $\mathfrak{A}$, we have that $\mathfrak{A}\in\mathsf{var}^\mathsf{G}(f)$ and $\mathsf{var}^\mathsf{G}(\mathsf{T}^{\mathsf{G}}(\mathfrak{A}))\subseteq\mathsf{var}^\mathsf{G}(f)$. Using this fact, the next result generalizes Corollary \ref{3.05} and is an answer to Problem \ref{3.01}.
	\begin{corollary}\label{3.32}
Let $\mathsf{G}$ be a finite abelian group, $\mathbb{F}$ an algebraically closed field of characteristic zero, and $\mathfrak{A}$ any associative $\mathsf{G}$-graded algebra. If $\mathfrak{A}_e$ satisfies a polynomial identity of degree $2$, then either $\mathfrak{A}$ is a nilpotent algebra or $\mathfrak{A}_e$ is a commutative algebra.
	\end{corollary}


\subsection{Graded rings with the central neutral components}\label{subsecringcentral}

Here, we denote by $\mathsf{S}$ a cancellative monoid (i.e. a monoid which satisfies $\xi\zeta=\tau\zeta$ iff $\xi=\tau$, and $\tilde\zeta \tilde\xi=\tilde\zeta \tilde\tau$ iff $\tilde\xi=\tilde\tau$, for any $\xi,\tilde\xi,\zeta,\tilde\zeta,\tau,\tilde\tau\in\mathsf{S}$), by $\mathfrak{R}$ an associative ring with a finite $\mathsf{S}$-grading $\Gamma$, and by $\mathcal{Z}(\mathfrak{R})$ the center of $\mathfrak{R}$. Let us present some results involving associative rings graded by a cancellative monoid whose neutral component is central. 
%

\begin{theorem}\label{3.06}
Let $\mathsf{S}$ be a cancellative monoid, and $\mathfrak{R}$ an associative ring with a finite $\mathsf{S}$-grading of order $d$. If $\mathfrak{R}_e$ is central in $\mathfrak{R}$ and $d\in\{1,2,3\}$, then $[x_1, \dots, x_{d+1}]\equiv0$ in $\mathfrak{R}$.
\end{theorem}
\begin{proof}
Firstly, by Lemma \ref{mardua3.9}, if $\mathfrak{R}_e=\{0\}$, then $\mathfrak{R}^{d+1}=\{0\}$. In particular, $[x_1,\dots, x_{d+1}]\equiv0$ in $\mathfrak{R}$ in this case.

Assume that $\mathfrak{R}_e\neq\{0\}$. For $d=1$, we have $\mathfrak{R}=\mathfrak{R}_e$, and hence, if $\mathfrak{R}_e\subseteq\mathcal{Z}(\mathfrak{R})$, then $\mathfrak{R}$ is commutative, i.e. $[x_1, x_2]\equiv0$ in $\mathfrak{R}$.

Suppose $d=2$ and put $\mathsf{Supp}(\Gamma)=\{e, \xi\}$, where $\xi\neq e$, then, either $\xi^2=e$ or $\xi^2\notin \mathsf{Supp}(\Gamma)$, because $\mathsf{S}$ is cancellative. Anyway, $(\mathfrak{R}_\xi)^2\subseteq \mathfrak{R}_e$. Given $a,b,c\in \mathfrak{R}$, we can write $a=a_e+a_\xi$, $b=b_e+b_\xi$. Since $\mathfrak{R}_e\subseteq \mathcal{Z}(\mathfrak{R})$, it follows that
	\begin{equation*}
[a, b, c]= [a_e+a_\xi, b_e+b_\xi, c]= [a_\xi, b_\xi, c]=[[a_\xi, b_\xi], c]=0
.
	\end{equation*}
Therefore, $[a, b, c]=0$ for any $a,b,c\in\mathfrak{R}$.

Now, assume $d=3$, and put $\mathsf{Supp}(\Gamma)=\{e, \xi, \zeta\}$. Consider the elements $\xi\zeta, \zeta\xi\in \mathsf{S}$. \textbf{Claim:} either $\zeta\xi=\xi\zeta=e$ or $\zeta\xi,\xi\zeta\notin\mathsf{Supp}(\Gamma)$. In fact, since $\mathsf{S}$ is cancellative, we have $\xi\zeta, \zeta\xi\notin\{\zeta,\xi\}$. Hence, if $\zeta\xi\in\mathsf{Supp}(\Gamma)$, then $\zeta\xi=e$, and hence, $\zeta\xi\zeta=\zeta$, and by cancellation law, it follows that $\xi\zeta=e$. Similarly, $\xi\zeta\in\mathsf{Supp}(\Gamma)$ implies $\xi\zeta=\zeta\xi=e$. Anyway, we have $\mathfrak{R}_\xi \mathfrak{R}_\zeta, \mathfrak{R}_\zeta \mathfrak{R}_\xi \subseteq \mathfrak{R}_e$.

Given $a,b,c\in \mathfrak{R}$, we can write $a=a_e+a_\xi+a_\zeta$, $b=b_e+b_\xi+b_\zeta$, and $c=c_e+c_\xi+c_\zeta$. Hence, since $\mathfrak{R}_e\subseteq \mathcal{Z}(\mathfrak{R})$ and $\mathfrak{R}_\xi \mathfrak{R}_\zeta, \mathfrak{R}_\zeta \mathfrak{R}_\xi \subseteq \mathfrak{R}_e$, we have that 
\begin{equation*}
	\begin{split}
[a, b, c] &= [a_e+a_\xi+a_\zeta, b_e+b_\xi+b_\zeta, c]= [a_\xi+a_\zeta, b_\xi+b_\zeta, c] \\
		  &= [a_\xi, b_\xi, c]+[a_\zeta, b_\zeta, c]+[a_\zeta, b_\xi, c]+[a_\xi, b_\zeta, c]  
		  =[a_\xi, b_\xi, c]+[a_\zeta, b_\zeta, c] \\
		  &= [a_\xi, b_\xi, c_e+c_\xi+c_\zeta]+[a_\zeta, b_\zeta, c_e+c_\xi+c_\zeta] 
		  = [a_\xi, b_\xi, c_\xi+c_\zeta]+[a_\zeta, b_\zeta, c_\xi+c_\zeta] \\
		  &= [a_\xi, b_\xi, c_\xi]+[a_\xi, b_\xi, c_\zeta]+[a_\zeta, b_\zeta, c_\xi]+[a_\zeta, b_\zeta, c_\zeta] \\
		  &= ([a_\xi, b_\xi, c_\xi]+[a_\zeta, b_\zeta, c_\zeta])+([a_\xi, b_\xi, c_\zeta]+[a_\zeta, b_\zeta, c_\xi]) .
	\end{split}
\end{equation*}
Notice that $[a_\xi, b_\xi, c_\zeta] +[a_\zeta, b_\zeta, c_\xi] =0$, since
%
%
%
\begin{equation*}
	\begin{split}
[a_\xi, b_\xi, c_\zeta]&=[a_\xi b_\xi- b_\xi a_\xi, c_\zeta]=[a_\xi b_\xi, c_\zeta]- [b_\xi a_\xi, c_\zeta]
		  = (a_\xi b_\xi) c_\zeta - c_\zeta (a_\xi b_\xi) -(b_\xi a_\xi)c_\zeta+c_\zeta (b_\xi a_\xi) \\
		  &= a_\xi b_\xi c_\zeta - (c_\zeta a_\xi) b_\xi -b_\xi a_\xi c_\zeta+(c_\zeta b_\xi) a_\xi
		  = a_\xi b_\xi c_\zeta - b_\xi(c_\zeta a_\xi) -b_\xi a_\xi c_\zeta+ a_\xi(c_\zeta b_\xi)\\
		  &= a_\xi b_\xi c_\zeta - (b_\xi c_\zeta) a_\xi -b_\xi a_\xi c_\zeta+(a_\xi c_\zeta) b_\xi
		  = a_\xi b_\xi c_\zeta - a_\xi (b_\xi c_\zeta) -b_\xi a_\xi c_\zeta+ b_\xi(a_\xi c_\zeta)=0 
	\end{split}
\end{equation*}
and
\begin{equation*}
	\begin{split}
[a_\zeta, b_\zeta, c_\xi] &=[a_\zeta b_\zeta-b_\zeta a_\zeta, c_\xi]=[a_\zeta b_\zeta, c_\xi]-[b_\zeta a_\zeta, c_\xi] 
		  = (a_\zeta b_\zeta) c_\xi - c_\xi (a_\zeta b_\zeta) - (b_\zeta a_\zeta) c_\xi+ c_\xi(b_\zeta a_\zeta) \\
		  &= a_\zeta b_\zeta c_\xi - (c_\xi a_\zeta) b_\zeta - b_\zeta a_\zeta c_\xi + (c_\xi b_\zeta) a_\zeta 
		  = a_\zeta b_\zeta c_\xi - b_\zeta (c_\xi a_\zeta) - b_\zeta a_\zeta c_\xi + a_\zeta (c_\xi b_\zeta)  \\
		  &= a_\zeta b_\zeta c_\xi - (b_\zeta c_\xi) a_\zeta - b_\zeta a_\zeta c_\xi + (a_\zeta c_\xi) b_\zeta  
		  = a_\zeta b_\zeta c_\xi - a_\zeta (b_\zeta c_\xi) - b_\zeta a_\zeta c_\xi + b_\zeta (a_\zeta c_\xi) = 0 \ .
	\end{split}
\end{equation*}
Hence, $[a, b, c] = [a_\xi, b_\xi, c_\xi]+[a_\zeta, b_\zeta, c_\zeta]$. Observe that $\xi^2\neq \xi$ and $\zeta^2\neq \zeta$, because $\xi\neq e$ and $\zeta\neq e$. And so $\mathfrak{R}_\xi \mathfrak{R}_\xi \subseteq \mathfrak{R}_e$ or $\mathfrak{R}_\xi \mathfrak{R}_\xi \subseteq\mathfrak{R}_\zeta$, and $\mathfrak{R}_\zeta \mathfrak{R}_\zeta \subseteq \mathfrak{R}_e$ or $\mathfrak{R}_\zeta \mathfrak{R}_\zeta \subseteq \mathfrak{R}_\xi$. If $\mathfrak{R}_\xi \mathfrak{R}_\xi \subseteq \mathfrak{R}_\zeta$, then $\xi^2=\zeta$ or $\xi^2\notin\mathsf{Supp}(\Gamma)$, and thus, either $\xi^2\notin\mathsf{Supp}(\Gamma)$, or $\xi^3\notin\mathsf{Supp}(\Gamma)$ or $\xi^3=e$, since $\xi\zeta, \zeta\xi\notin\{\zeta,\xi\}$. Consequently, we deduce that either $(\mathfrak{R}_\xi)^2\subseteq\mathfrak{R}_e$ or $(\mathfrak{R}_\xi)^3\subseteq\mathfrak{R}_e$. From this, either $[a_\xi, b_\xi, c_\xi]=0$ or $[a_\xi, b_\xi, c_\xi]\in\mathfrak{R}_e$, for any $a, b, c \in \mathfrak{R}$. Analogously, we deduce that either $[a_\zeta, b_\zeta, c_\zeta]=0$ or $[a_\zeta, b_\zeta, c_\zeta]\in\mathfrak{R}_e$, for any $a, b, c \in \mathfrak{R}$. In any case, $[a_\xi, b_\xi, c_\xi],[a_\zeta, b_\zeta, c_\zeta]\in\mathfrak{R}_e$. Therefore, $[a,b,c]\in \mathfrak{R}_e$ for $a, b, c\in \mathfrak{R}$, and so $[a, b, c, d]=0$ for any $a, b, c, d\in\mathfrak{R}$. The result follows.
\end{proof}
   
From the proof above, observe that $[\mathfrak{R},\mathfrak{R},\mathfrak{R}]\subseteq\mathcal{Z}(\mathfrak{R})$, and so $[x,y,z]$ is a central polynomial for $\mathfrak{R}$. On the other hand, in Theorem \ref{3.06}, if $\mathsf{S}=\mathbb{Z}_2$ (resp. $\mathsf{S}=\mathbb{Z}_3$), then any $\mathsf{S}$-graded ring $\mathfrak{R}$ with the central neutral component satisfies the polynomial identity $[x_1,x_2,x_3]=0$ (resp. $[x_1,x_2,x_3,x_4]=0$).

\begin{corollary}\label{3.07}
Let $\mathsf{S}$ be a group and $\mathfrak{R}$ a ring with an $\mathsf{S}$-grading $\Gamma$. Let $P$ be a normal subgroup of $\mathsf{S}$, and $\overline{\Gamma}: \mathfrak{R}=\bigoplus_{\bar{\xi}\in\mathsf{S}/P} \mathfrak{R}_{\bar{\xi}}$ the $\mathsf{S}/P$-grading induced by $\Gamma$. Suppose that $\overline{\Gamma}$ has a finite support of order $d$. If $\mathfrak{R}_{\bar{e}}=\bigoplus_{p\in P} \mathfrak{R}_p\subseteq\mathcal{Z}(\mathfrak{R})$ and $d\in\{1,2,3\}$, then $[x_1,\dots,x_{d+1}]\equiv0$ in $\mathfrak{R}$.
\end{corollary}

It is important to note that $\Gamma$ in the previous corollary is not necessarily a finite $\mathsf{S}$-grading.

Below, we exhibit two examples which show that the condition ``$d\in\{1,2,3\}$'' in Theorem \ref{3.06} is necessary, where $d$ is the order of the support of $\mathsf{G}$-grading on $\mathfrak{R}$. Anyway, we show that Theorem \ref{3.06}, in general, does not hold when $d\geq4$.

\begin{example}
	Let $\mathcal{K}=\mathbb{Z}_2\times\mathbb{Z}_2$ be the Klein group and $\mathfrak{B}=M_2(\mathbb{F})$ the algebra of matrices of order $2$ over $\mathbb{F}$ of Example \ref{3.17}, with its natural $\mathcal{K}$-grading.  
	Notice that $\mathfrak{B}$ satisfies the $\mathcal{K}$-graded polynomial identities $[x^{(e)},y^{(\xi)}]$ for any $\xi\in \mathcal{K}$, where $e=(\bar0,\bar0)$ is the neutral element of $\mathcal{K}$, but $[x_1,x_2,\dots,x_n]$ is not a (ordinary) polynomial identity for $\mathfrak{B}$, for all $n\in\mathbb{N}$, since 
	\begin{equation*}
	[E_{12},\underbrace{E_{22},\dots, E_{22}}_{(n-1)-times}]=E_{12}\neq0 \ ,
	\end{equation*}
for all $n\in\mathbb{N}$.
\end{example}

\begin{example}\label{3.27}
Let $\mathcal{K}=\mathbb{Z}_2\times\mathbb{Z}_2$, $\mathbb{F}$ a field with $\mathsf{char}(\mathbb{F})\neq2$, and $\mathbb{H}$ a \textit{Quaternion algebra} over $\mathbb{F}$, i.e. $\mathbb{H}=\{a1+b\mathsf{i}+c\mathsf{j}+d\mathsf{k} : a,b,c,d\in\mathbb{F}\}=\mathbb{F}\left( \mathsf{i}, \mathsf{j}, \mathsf{k}\right)$, where $\mathsf{i}^2=\mathsf{j}^2=\mathsf{k}^2=-1$, and $\mathsf{i}\mathsf{j}=-\mathsf{j}\mathsf{i}=\mathsf{k}$, and $1$ is the unity. The algebra $\mathbb{H}$ has a natural $\mathcal{K}$-grading given by $\mathbb{H}=\mathbb{H}_{(\bar0,\bar0)}\oplus\mathbb{H}_{(\bar0,\bar1)}\oplus\mathbb{H}_{(\bar1,\bar0)}\oplus\mathbb{H}_{(\bar1,\bar1)}$, where $\mathbb{H}_{(\bar0,\bar0)}=\mathsf{span}_\mathbb{F}\{1 \}$, $\mathbb{H}_{(\bar0,\bar1)}=\mathsf{span}_\mathbb{F}\{\mathsf{i} \}$, $\mathbb{H}_{(\bar1,\bar0)}=\mathsf{span}_\mathbb{F}\{\mathsf{j} \}$ and $\mathbb{H}_{(\bar1,\bar1)}=\mathsf{span}_\mathbb{F}\{\mathsf{k} \}$. Obviously $\mathbb{H}_{(\bar0,\bar0)}$ is central in $\mathbb{H}$, but $\mathbb{H}$ is not a nilpotent algebra, since $\mathbb{H}$ is a division algebra. Moreover, $\mathbb{H}$ does not satisfy the identity polynomial $[x_1, x_2,\dots,x_n]$, for all $n\in\mathbb{N}$, $n\geq2$, since 
	\begin{equation*}
	[\mathsf{i}, \underbrace{\mathsf{j}, \mathsf{j},\dots, \mathsf{j}}_{(n-1)-times}]=2^{(n-1)}\mathsf{i}\mathsf{j}^{(n-1)}=\pm 2^{(n-1)} \cdot 
		\begin{cases}
     \mathsf{k} \ ,& \mbox{if } n \mbox{ is even} \\
     \mathsf{i} \ ,& \mbox{if } n \mbox{ is odd}  
		\end{cases} \ .
	\end{equation*}
\end{example}

In the next proposition, let us use Example \ref{3.27} to build an algebra $\mathfrak{A}$ with a finite $\mathsf{S}$-grading of order $d\geq4$ such that its neutral component is central, but $[x_1,\dots,x_m]$ is not its polynomial identity for $\mathfrak{A}$, for all $m\in\mathbb{N}$.

\begin{proposition}\label{3.28}
For all integer $d\geq4$, there exists an algebra $\mathfrak{A}$ with an $\mathsf{S}$-grading of order $d$ such that $\mathfrak{A}_e$ is central in $\mathfrak{A}$, but the polynomial $[x_1,\dots,x_m]$ is not an identity for $\mathfrak{A}$, for all $m\in\mathbb{N}$.
\end{proposition}
\begin{proof}
Let $\mathbb{H}$ be the Quaternion algebra of Example \ref{3.27}. Now, consider that $\mathfrak{B}=\mathsf{span}_{\mathbb{F}} \{x\}$ is a nilpotent algebra, where $x\neq0$ and $x^2=0$. Note that  $\mathfrak{B}=\mathfrak{B}_{\bar{1}}$ is a $\mathbb{Z}_2$-grading on $\mathfrak{B}$. Take the algebra $\mathfrak{A}_1=\mathbb{H}\times\mathfrak{B}$ (the direct product of the algebras $\mathbb{H}$ and $\mathfrak{B}$) with a $(\mathbb{Z}_2)^3$-grading $\Gamma_1$ defined by 
$(\mathfrak{A}_1)_{(\bar{i},\bar{j},\bar{l})}=\mathbb{H}_{(\bar{i},\bar{j})}\times\mathfrak{B}_{\bar{l}}$, for all $\bar{i}, \bar{j}, \bar{l}\in\mathbb{Z}_2$. The support of $\Gamma_1$ has order $5$, and $(\mathfrak{A}_1)_{e_1}$ is central in $\mathfrak{A}_1$, where $e_1=(\bar0,\bar0,\bar0)$. Since $\mathbb{H}\cong\mathbb{H}\times\{0\}\subset\mathfrak{A}_1$, it follows from Example \ref{3.27} that $[x_1,\dots,x_{m}]\not\equiv0$ in $\mathfrak{A}_1$ for all $m\in\mathbb{N}$, $m\geq2$. 
Now, take the algebra $\mathfrak{A}_2=\mathbb{H}\times\mathfrak{B}^2=\mathfrak{A}\times\mathfrak{B}\times\mathfrak{B}$ with a $(\mathbb{Z}_2)^{4}$-grading $\Gamma_2$ induced by gradings of $\mathbb{H}$ and $\mathfrak{B}$, i.e. $(\mathfrak{A}_2)_{(\bar{i_1},\bar{i_2},\bar{i_3},\bar{i_4})}=\mathbb{H}_{(\bar{i_1},\bar{i_2})}\times\mathfrak{B}_{\bar{i_3}}\times\mathfrak{B}_{\bar{i_4}}$, for all $\bar{i_1},\bar{i_2},\bar{i_3},\bar{i_4}\in\mathbb{Z}_2$. Observe that $(\mathfrak{A}_2)_{e_2}$ is central in $\mathfrak{A}_2$, where $e_2=(\bar0,\bar0,\bar0,\bar0)$, and $|\mathsf{Supp}(\Gamma_2)|=6$. Since $\mathbb{H}\cong\mathbb{H}\times\{0\}\times\{0\}\subset\mathfrak{A}_2$, it follows that $\mathfrak{A}_2$ does not satisfy the identity $[x_1, x_2,\dots,x_m]$ for all $m\in\mathbb{N}$, $m\geq2$. 
By repeating this process, we can build the algebra $\mathfrak{A}_n=\mathbb{H}\times\mathfrak{B}^n=\mathfrak{A}\times\underbrace{\mathfrak{B}\times\cdots\times\mathfrak{B}}_{n-times}$ with a $(\mathbb{Z}_2)^{n+2}$-grading $\Gamma_n$ induced by gradings of $\mathbb{H}$ and $\mathfrak{B}$ such that $(\mathfrak{A}_n)_{e_n}$ is central in $\mathfrak{A}_n$, where $e_n$ is the neutral element of $(\mathbb{Z}_2)^{n+2}$, $|\mathsf{Supp}(\Gamma_n)|=n+4$ and $\mathfrak{A}_n$ does not satisfy the identity $[x_1, x_2,\dots,x_m]$ for all $m\in\mathbb{N}$, $m\geq2$. Furthermore, the proposition follows.
\end{proof}

Although the previous proposition works for all $d\geq4$, it does not cover the cases $\mathsf{S}=\mathbb{Z}_p$'s, i.e. when $\mathsf{S}$ is a finite cyclic group. Therefore, we have the following problem.

	\vspace{2mm}
	\noindent{\bf Problem$^*$.} 
\textit{For any $p\geq4$, is there some $\mathbb{Z}_p$-graded ring $\mathfrak{R}$ with $\mathfrak{R}_{\bar0}\subseteq\mathcal{Z}(\mathfrak{R})$ such that $[x_1,\dots,x_n]\not\equiv0$ in $\mathfrak{R}$ for all $n\in\mathbb{N}$? Conversely, if $\mathfrak{R}$ is a $\mathbb{Z}_p$-graded ring such that $\mathfrak{R}_{\bar0}\subset\mathcal{Z}(\mathfrak{R})$, $p\geq4$, is it $[x_1,\dots,x_m]\equiv0$ in $\mathfrak{R}$ for some $m\in\mathbb{N}$?}
	\vspace{2mm}

In language of Lie algebras, given an associative algebra $\mathfrak{A}$, the pair $(\mathfrak{A},[\ ,\ ])$, denoted by $\mathfrak{A}^{(-)}$, has naturally a structure of Lie algebra, and so, we have the question: if $\mathsf{ad} x$ is a zero homomorphism for any $x\in\mathfrak{A}_e$ (i.e. given any $x\in\mathfrak{A}_e$, $[x,y]= 0$ for any $y\in\mathfrak{A}$), then is $\mathfrak{A}^{(-)}$ a nilpotent Lie algebra? Recall that a Lie algebra $\mathfrak{L}$ is called \textit{nilpotent} if it satisfies some $n$th commutator, i.e. $[x_1,\dots,x_{n}]\equiv0$ in $\mathfrak{L}$ for some $n\in\mathbb{N}$. In the study of Lie algebras, other important concept is that of \textit{solvable Lie algebra}, which is intrinsically related to nilpotent Lie algebras. Recall also that a Lie algebra $\mathfrak{L}$ is called \textit{solvable} if $\mathfrak{L}^{(k)}=\{0\}$ for some $k\in\mathbb{N}$, where $\mathfrak{L}^{(k)}$ is inductively defined by: $\mathfrak{L}'=[\mathfrak{L}^{(k)},\mathfrak{L}^{(k)}]$, $\mathfrak{L}^{(2)}=[\mathfrak{L}',\mathfrak{L}']$, and $\mathfrak{L}^{(k)}=[\mathfrak{L}^{(k-1)},\mathfrak{L}^{(k-1)}]$ for all $k>2$. It is well known that any nilpotent Lie algebra is a solvable Lie algebra. The converse is not true. Another result well known is that, in characteristic zero, any finite dimensional Lie algebra $\mathfrak{L}$ is solvable iff its derived subalgebra $[\mathfrak{L},\mathfrak{L}]$ is a nilpotent Lie algebra (see \cite{Serr09}, Corollary 5.3 and its Remark, p.19, or \cite{Jaco13}, Corollary 1, p.51). For further reading, as well as an overview, on Lie algebras, we suggest the works \cite{Hump12}, \cite{Jaco13} and \cite{Serr09}. 

In this sense, the Problem$^*$ can be rewritten as follows:
\begin{problem}\label{3.29}
If $\mathfrak{A}$ is a $\mathsf{G}$-graded algebra such that its neutral component $\mathfrak{A}_e$ is central, then is $\mathfrak{A}^{(-)}$ a solvable/nilpotent Lie algebra? And about the commutator ideal of $\mathfrak{A}$, is it a nilpotent algebra? 
\end{problem}


We recall that \textit{the commutator ideal} of an algebra $\mathfrak{A}$ is the (two-sided) ideal of $\mathfrak{A}$ generated by $[\ ,\ ]$, i.e. the ideal of $\mathfrak{A}$ generated by all the elements $[a,b]$, $a,b\in\mathfrak{A}$. Note that the commutator ideal of $\mathfrak{A}$ is equal to $\mathsf{span}_{\mathbb{F}}\{c[a,b]d\in\mathfrak{A}: a,b,c,d\in\mathfrak{A}\}$.

In what follows, we will answer affirmatively the Problem \ref{3.29} for any algebra over a field of characteristic zero with a grading by a cyclic group $\mathsf{G}$ of odd order. When $\mathsf{G}$ has even order, Problem \ref{3.29} also has a positive answer for finitely generated algebras. The Example \ref{3.27} shows that the condition ``$\mathsf{G}$ a finite cyclic group'' in the next theorems is indeed necessary. Inclusive, the commutator ideal of the algebra $\mathbb{H}$ in Example \ref{3.27} is not nilpotent.

\begin{remark}\label{3.25}
Let $\mathsf{G}$ be a finite cyclic group and $\mathbb{F}$ an arbitrary field. Since any $2$-cocycle $\sigma$ on $\mathsf{G}$ is symmetric, because it is a $2$-coboundary (for a proof of this fact, see \cite{MarduaThesis}, Corollary 1.2.8 , p.28), it follows that $\mathbb{F}^{\sigma}[\mathsf{H}]$ is a commutative algebra, for any subgroup $\mathsf{H}$ of $\mathsf{G}$ and $2$-cocycle $\sigma$ on $\mathsf{H}$.
\end{remark}

\begin{lemma}\label{3.14}
Let $\mathsf{G}$ be a finite cyclic group, $\mathbb{F}$ an arbitrary field, and $\mathfrak{A}=\mathfrak{B} \oplus \mathsf{J}(\mathfrak{A})$ a finite dimensional $\mathsf{G}$-graded algebra, where $\mathsf{J}(\mathfrak{A})$ is Jacobson radical of $\mathfrak{A}$ and $\mathfrak{B}=M_k(\mathbb{F}^\sigma [\mathsf{H}])$ with an elementary-canonical $\mathsf{G}$-grading. If $\mathfrak{A}_e$ is central in $\mathfrak{A}$, then $[\mathfrak{A},\mathfrak{A}]^n=\{0\}$ for some $n\in\mathbb{N}$. In addition, the commutator ideal of $\mathfrak{A}$ is nilpotent.
\end{lemma}
\begin{proof}
Since $\mathfrak{B}_e\subseteq\mathfrak{A}_e$, we have that $\mathfrak{B}_e$ is central in $\mathfrak{A}$, and so, $\mathfrak{B}_e$ is central in $\mathfrak{B}$. As $\mathfrak{B}$ is graded with an elementary-canonical $\mathsf{G}$-grading, by item \textit{ii)} from Lemma \ref{3.04}, it follows that $\mathfrak{B}=\mathbb{F}^\sigma [\mathsf{H}]$. Now, by Remark \ref{3.25}, we conclude that $\mathfrak{B}$ is commutative. From this, for any $a,b\in\mathfrak{B}$ and $x,y\in\mathsf{J}(\mathfrak{A})$, we must have 
$[a+x,b+y]=[a,b]+[a,y]+[x,a]+[x,y]=[a,y]+[x,a]+[x,y]\in\mathsf{J}(\mathfrak{A})$, since $[a,b]=0$ and $[a,y], [x,a], [x,y]\in\mathsf{J}(\mathfrak{A})$, because $\mathsf{J}(\mathfrak{A})$ is a (two-sided) ideal of $\mathfrak{A}$. Therefore, $[\mathfrak{A},\mathfrak{A}]\subseteq\mathsf{J}(\mathfrak{A})$. Finally, as $\mathsf{J}(\mathfrak{A})$ is nilpotent, the result follows.
\end{proof}

The proof's argument of the previous lemma can be easily extended to a finite dimensional algebra $\mathfrak{A}=\mathfrak{B}\oplus\mathsf{J}(\mathfrak{A})$, where $\mathfrak{B}=M_{k_1}(\mathbb{F}^{\sigma_1}[\mathsf{H}_1]) \times \cdots \times M_{k_p}(\mathbb{F}^{\sigma_p}[\mathsf{H}_p])$. More than that, let us extend Lemma \ref{3.14} to graded algebras that are finitely generated.

\begin{theorem}\label{3.09}
Let $\mathbb{F}$ be an algebraically closed field of characteristic zero, $\mathsf{G}$ a finite cyclic group, $\mathfrak{A}$ a  finitely generated $\mathbb{F}$-algebra with a $\mathsf{G}$-grading. If $\mathfrak{A}_e$ is central in $\mathfrak{A}$, then $\mathfrak{A}^{(-)}$ is a solvable Lie algebra. In addition, the commutator ideal of $\mathfrak{A}$ is nilpotent.
\end{theorem}
\begin{proof}
First, by Lemma \ref{teoBergCoheBaht}, we have that $\mathfrak{A}$ is a $PI$-algebra. By Lemmas \ref{teoIrina01} and \ref{teoIrina03}, there exists an algebra $\widetilde{\mathfrak{A}} = \left(M_{k_1}(\mathbb{F}^{\sigma_1}[\mathsf{H}_1]) \times \cdots \times M_{k_p}(\mathbb{F}^{\sigma_p}[\mathsf{H}_p])\right) \oplus \mathsf{J}(\mathfrak{A})$, as in \eqref{3.08}, such that $\mathfrak{A}$ and $\widetilde{\mathfrak{A}}$ have the same $\mathsf{G}$-graded polynomial identities. Put $\mathfrak{B}=M_{k_1}(\mathbb{F}^{\sigma_1}[\mathsf{H}_1]) \times \cdots \times M_{k_p}(\mathbb{F}^{\sigma_p}[\mathsf{H}_p])$, where each $M_{k_i}(\mathbb{F}^{\sigma_i}[\mathsf{H}_i])$ is graded with an elementary-canonical $\mathsf{G}$-grading. By Lemma \ref{3.14} and its proof, it is easy to see that $\mathfrak{B}=\mathbb{F}^{\sigma_1}[\mathsf{H}_1] \times \cdots \times \mathbb{F}^{\sigma_p}[\mathsf{H}_p]$ is a commutative algebra and $[\widetilde{\mathfrak{A}},\widetilde{\mathfrak{A}}]\subseteq \mathsf{J}(\mathfrak{A})$. Therefore, $\mathfrak{A}^{(-)}$ is a solvable Lie algebra, and $[\ ,\ ]$ generates a nilpotent ideal of $\mathfrak{A}$.
\end{proof}

It is interesting to comment that, being $\mathfrak{A}$ an algebra as in Lemma \ref{3.14} and $\mathbb{F}$ as in Theorem \ref{3.09}, not necessarily $\mathfrak{B}\subset\mathcal{Z}(\mathfrak{A})$. In fact, by \cite{Mardua02}, Corollary 3.16 and Theorem 5.2, we have that $\mathsf{J}(\mathfrak{A})=\mathsf{J}_{00}\oplus \mathsf{J}_{11}$, where $\mathsf{J}_{00}$ is a $0$-$\mathfrak{B}$-bimodule and $\mathsf{J}_{11}$ is a faithful unitary $\mathfrak{B}$-bimodule. Now, again by Theorem 5.2 in \cite{Mardua02}, item \textit{(v)} (see also Theorem 4.6 and its proof), we have that $\mathsf{J}_{11}=\mathfrak{B}\mathsf{N}$ for some nilpotent graded algebra $\mathsf{N}=\mathsf{span}_\mathbb{F}\{d_1,\dots,d_s\}$, where $d_i \eta_\xi=\gamma_{i,\xi}\eta_{\xi} d_i$, with $\gamma_{i,\xi}\in\mathbb{F}$, and each $\gamma_{i,\xi}$ is associated with some irreducible character $\chi$ of $\mathsf{H}$. Therefore, we can not ensure that $bx=xb$ for any $b\in\mathfrak{B}$ and $x\in\mathsf{J}$.

From now on, let us weaken the condition ``\textit{$\mathbb{F}$ is an algebraically closed field}'' which is required in Theorem \ref{3.09}.

Let $\mathfrak{A}$ be an $\mathbb{F}$-algebra with a $\mathsf{G}$-grading $\Gamma$, and $\mathbb{K}\supseteq\mathbb{F}$ an extension of fields. Consider the $\mathbb{K}$-algebra $\overline{\mathfrak{A}}=\mathfrak{A}\otimes_{\mathbb{F}}\mathbb{K}$ given by the tensor product of $\mathbb{F}$-algebras $\mathfrak{A}$ and $\mathbb{K}$. We have that $\overline{\mathfrak{A}}=\mathfrak{A}\otimes_{\mathbb{F}}\mathbb{K}$ is a $\mathbb{K}$-algebra (and also an $\mathbb{F}$-algebra) with a $\mathsf{G}$-grading induced by $\Gamma$ defined by $\overline{\mathfrak{A}}_\xi=\mathfrak{A}_\xi\otimes_{\mathbb{F}}\mathbb{K}$ (as $\mathbb{K}$-spaces) for any $\xi\in\mathsf{G}$. In this sense, naturally $\mathsf{T}^{\mathsf{G}}(\mathfrak{A})\subset\mathbb{F}\langle X^{\mathsf{G}}\rangle$ and $\mathsf{T}^{\mathsf{G}}(\overline{\mathfrak{A}})\subset\mathbb{K}\langle X^{\mathsf{G}}\rangle$. Note that $\mathfrak{A}$ can be seen as a graded $\mathbb{F}$-subalgebra of $\overline{\mathfrak{A}}$, since $\overline{\mathfrak{A}}$ is also an $\mathbb{F}$-algebra, via map $a\mapsto a\otimes1$ for any $a\in\mathfrak{A}$, and so $\mathfrak{A}\cong_{\mathsf{G}} \mathfrak{A}\otimes_\mathbb{F} \mathbb{F} \subset\overline{\mathfrak{A}}$ (as $\mathbb{F}$-algebras). Consequently, given any  $f\in\mathbb{F}\langle X^{\mathsf{G}} \rangle$, if $f\equiv_{\mathsf{G}}0$ in $\overline{\mathfrak{A}}$, then $f\equiv_{\mathsf{G}}0$ in $\mathfrak{A}$. Therefore, we have that $\mathsf{T}^{\mathsf{G}}(\overline{\mathfrak{A}})\bigcap\mathbb{F}\langle X^{\mathsf{G}}\rangle$ is contained in $\mathsf{T}^{\mathsf{G}}(\mathfrak{A})$.

Given a graded polynomial $g\in\mathbb{F}\langle X^{\mathsf{G}} \rangle$, we write $g=\sum_{\xi\in\mathsf{G}}g_\xi$, where each $g_\xi$ is the homogeneous graded polynomial of $\mathbb{F}\langle X^{\mathsf{G}} \rangle$ formed by the sum of all the homogeneous graded monomials of $g$ of degree $\xi$ (see Definition \ref{3.26}). Each $g_\xi$ is called ``\textit{a homogeneous component of degree $\xi$ of $g$}'', or simply ``\textit{a $\mathsf{G}$-homogeneous component of $g$}''.

\begin{lemma}\label{3.10}
Let $\mathfrak{A}$ be a $\mathsf{G}$-graded algebra and $g\in\mathbb{F}\langle X^{\mathsf{G}}\rangle$ a graded polynomial. Suppose $g=\sum_{\xi\in \mathsf{G}} g_\xi$, where $g_\xi\in(\mathbb{F}\langle X^{\mathsf{G}}\rangle)_\xi$, $\xi\in\mathsf{G}$. Then $g\equiv_{\mathsf{G}}0$ in $\mathfrak{A}$ iff $g_\xi\equiv_{\mathsf{G}}0$ in $\mathfrak{A}$ for any $\xi\in\mathsf{G}$.
\end{lemma}
\begin{proof}
Clearly $g_\xi\equiv_{\mathsf{G}}0$ in $\mathfrak{A}$ for any $\xi\in\mathsf{G}$ implies $g\equiv_{\mathsf{G}}0$ in $\mathfrak{A}$. Conversely, suppose $g\equiv_{\mathsf{G}}0$ in $\mathfrak{A}$. Put $g_\xi=g_\xi(x_1^{(\xi_1)},\dots,x_n^{(\xi_n)})$, $\xi\in\mathsf{G}$. Hence, we have that $g_\xi(a_{\xi_1},\dots,a_{\xi_n})\in\mathfrak{A}_\xi$ for any $a_{\xi_i}\in\mathfrak{A}_{\xi_i}$, $i=1,\dots, n$. So, by definition of $\mathsf{G}$-grading, it follows that $g_\xi\equiv_{\mathsf{G}}0$ in $\mathfrak{A}$ for any $\xi\in\mathsf{G}$.
\end{proof}

As defined in \cite{GiamZaic05}, Definition 1.3.1, p.5, a polynomial $g=g(x_1,\dots,x_n)$ in the variables $x_1,\dots,x_n$ is said to be \textit{homogeneous in the variable $x_s$} if $x_s$ appears with the same degree (number of times) in every monomials of $g$. If $g$ is homogeneous in the variables $x_1,\dots,x_n$, then we say ``\textit{$g$ is a multihomogeneous polynomial}''. 
 It is worth noting that ``homogeneous'' here differs (subtly) from ``homogeneous'' in Definition \ref{3.26}. 
So, we say that a graded polynomial $f=f(x_1^{(\xi_1)},\dots,x_n^{(\xi_n)})$ of $\mathbb{F}\langle X^{\mathsf{G}} \rangle$ is ``\textit{multihomogeneous and $\mathsf{G}$-homogeneous of degree $\xi$}'' (or still ``\textit{a multihomogeneous $\mathsf{G}$-homogeneous polynomial}'') if $f$ is multihomogeneous in the variables $x_1^{(\xi_1)},\dots,x_n^{(\xi_n)}$, and $f$ is $\mathsf{G}$-homogeneous of degree $\xi\in\mathsf{G}$. 
Obviously, if $w=\sum_{\xi\in\mathsf{G}}w_\xi$ is a graded polynomial of $\mathbb{F}\langle X^{\mathsf{G}} \rangle$, where $w_\xi$'s are the $\mathsf{G}$-homogeneous components of $w$, then each $w_\xi$ can be written as a sum of multihomogeneous $\mathsf{G}$-homogeneous graded polynomials (of degree $\xi$) (see the beginning of page 6, in \cite{GiamZaic05}).

\begin{lemma}\label{3.31}
Let $\mathbb{F}$ be an infinite field, $\mathsf{G}$ a group, $\mathfrak{A}$ a $\mathsf{G}$-graded $\mathbb{F}$-algebra. If $g\in \mathbb{F}\langle X^{\mathsf{G}}\rangle$ is a graded polynomial identity for $\mathfrak{A}$, then every multihomogeneous $\mathsf{G}$-homogeneous component of $g$ is still a graded identity for $\mathfrak{A}$.
\end{lemma}
\begin{proof}
By Lemma \ref{3.10}, we can assume that $g=g_\xi\in (\mathbb{F}\langle X^{\mathsf{G}}\rangle)_\xi$ is a homogeneous graded polynomial of degree $\xi\in\mathsf{G}$. Hence, the proof is adapted from the proof of Theorem 1.3.2, p.6, in \cite{GiamZaic05}.
\end{proof}

Notice that the converse of the previous lemma is still true. Let us show now the proposition below as a consequence of the lemmas above.
\begin{proposition}\label{3.11}
Let $\mathbb{F}$ be an infinite field, $\mathsf{G}$ a group, and $\mathfrak{A}$ a $\mathsf{G}$-graded $\mathbb{F}$-algebra. Let $\mathfrak{C}$ be a commutative $\mathbb{F}$-algebra and $\overline{\mathfrak{A}}\coloneqq\mathfrak{A}\otimes_{\mathbb{F}}\mathfrak{C}$ the tensor $\mathbb{F}$-algebra of $\mathfrak{A}$ and $\mathfrak{C}$ with $\mathsf{G}$-grading defined by $\overline{\mathfrak{A}}_\xi=\mathfrak{A}_\xi\otimes_{\mathbb{F}}\mathfrak{C}$ for any $\xi\in\mathsf{G}$. Every graded polynomial identity of $\mathfrak{A}$ is still a graded identity for $\mathfrak{A}\otimes_{\mathbb{F}}\mathfrak{C}$. In particular, if $\mathbb{K}\supseteq \mathbb{F}$ is an extension of fields, then $\mathfrak{A}$ and $\mathfrak{A}\otimes_{\mathbb{F}}\mathbb{K}$ satisfy the same graded polynomial identities in $\mathbb{F}\langle X^{\mathsf{G}}\rangle$.
\end{proposition}
\begin{proof}
Let $g\in \mathbb{F}\langle X^{\mathsf{G}}\rangle$ be a graded polynomial identity for $\mathfrak{A}$. By Lemmas  \ref{3.10} and \ref{3.31}, we can assume that $g$ is a multihomogeneous $\mathsf{G}$-homogeneous graded polynomial. Hence, the proof of the first part of the proposition is similar to the proof of Lemma 1.4.2, p.10, in \cite{GiamZaic05}.

On the other hand, since $\mathfrak{A}\cong_{\mathsf{G}}\mathfrak{A}\otimes_{\mathbb{F}}\mathbb{F}$, we can see $\mathfrak{A}$ as a $\mathsf{G}$-graded $\mathbb{F}$-subalgebra of $\mathfrak{A}\otimes_{\mathbb{F}}\mathbb{K}$, and hence, any graded polynomial identity $g\in\mathbb{F}\langle X^{\mathsf{G}}\rangle$ of $\mathfrak{A}\otimes_{\mathbb{F}}\mathbb{K}$ belongs to $\mathsf{T^G}(\mathfrak{A})$.
\end{proof}

Finally, let us conclude this work with two results that generalize Theorem \ref{3.09} for algebras on fields which are not necessarily algebraically closed.

\begin{theorem}\label{3.12}
Let $\mathbb{F}$ be a field of characteristic zero, $\mathsf{G}$ a finite cyclic group, $\mathfrak{A}$ a finitely generated $\mathbb{F}$-algebra with a $\mathsf{G}$-grading $\Gamma$. Suppose that $\mathfrak{A}_e$ is central in $\mathfrak{A}$. The commutator ideal of $\mathfrak{A}$ is nilpotent. Moreover, $\mathfrak{A}^{(-)}$ is a solvable Lie algebra. In addition, if the support of $\Gamma$ has at most 3 elements, then $\mathfrak{A}^{(-)}$ is a nilpotent Lie algebra.
\end{theorem}
\begin{proof}
First, the last part of the theorem follows from Theorem \ref{3.06}. Now, let $\mathbb{K}\supseteq\mathbb{F}$ be an extension of fields with $\mathbb{K}$ algebraically closed. Consider the tensor algebra $\overline{\mathfrak{A}}=\mathfrak{A}\otimes_{\mathbb{F}}\mathbb{K}$ with the $\mathsf{G}$-grading defined by $\overline{\mathfrak{A}}_\xi=\mathfrak{A}_\xi\otimes_{\mathbb{F}}\mathbb{K}$, $\xi\in\mathsf{G}$. 
Assume that $[x^{(e)},y^{(\xi)}]\equiv_{\mathsf{G}}0$ in $\mathfrak{A}$ for any $\xi\in\mathsf{G}$. 
By Proposition \ref{3.11}, it follows that $[x^{(e)},y^{(\xi)}]\equiv_{\mathsf{G}}0$ in $\overline{\mathfrak{A}}$ for any $\xi\in\mathsf{G}$, and so $\overline{\mathfrak{A}}_e$ is central in $\overline{\mathfrak{A}}$. 

On the other side, since $\mathfrak{A}$ is a finitely generated $\mathbb{F}$-algebra, we have that $\overline{\mathfrak{A}}$ is a finitely generated $\mathbb{K}$-algebra, because if $S$ generates $\mathfrak{A}$ as an $\mathbb{F}$-algebra, then $\{a\otimes 1_{\mathbb{K}} : a\in S\}$ generates $\mathfrak{A}\otimes_{\mathbb{F}}\mathbb{K}$ as a $\mathbb{K}$-algebra. 
Therefore, the result follows from Theorem \ref{3.09} and its proof, and because $\mathfrak{A}$ is an $\mathbb{F}$-subalgebra of $\overline{\mathfrak{A}}$.
\end{proof}

Now, using the idea of the proof of Proposition 5.4 in \cite{Mardua02}, we can improve Theorem \ref{3.12} by eliminating the requirement for $\mathfrak{A}$ to be a finitely generated algebra, but with $\mathsf{gcd}(|\mathsf{G}|,2)=1$.

\begin{theorem}\label{3.30}
Let $\mathbb{F}$ be a field of characteristic zero, $\mathsf{G}$ a cyclic group of odd order, $\mathfrak{A}$ an $\mathbb{F}$-algebra with a $\mathsf{G}$-grading. If $\mathfrak{A}_e$ is central in $\mathfrak{A}$, then the commutator ideal of $\mathfrak{A}$ is nilpotent. Consequently, $\mathfrak{A}^{(-)}$ is a solvable Lie algebra. In addition, $[x_1,x_2][x_3,x_4]\cdots[x_{2n-1},x_{2n}]\equiv0$ in $\mathfrak{A}$, for some $n\in\mathbb{N}$.
\end{theorem}
\begin{proof}
By the proof of Theorem \ref{3.12}, we can assume, without loss of generality, that $\mathbb{F}$ is an algebraically closed field. By Lemma \ref{teoBergCoheBaht}, $\mathfrak{A}$ is a $PI$-algebra, and so by Lemmas \ref{teoIrina02} and \ref{teoIrina03}, it follows that there exists a $\mathsf{G}\times\mathbb{Z}_2$-graded finite dimensional algebra $\widehat{\mathfrak{A}} = \mathfrak{B}\oplus\mathsf{J}$ such that $\mathsf{T^G}(\mathfrak{A})=\mathsf{T^G}(\mathsf{E}^\mathsf{G}(\widehat{\mathfrak{A}}))$, where $\mathsf{J}=\mathsf{J}(\widehat{\mathfrak{A}})$ is the Jacobson radical of $\widehat{\mathfrak{A}}$, and $\mathfrak{B}=M_{k_1}(\mathbb{F}^{\sigma_1}[\mathsf{H}_1]) \times \cdots \times M_{k_q}(\mathbb{F}^{\sigma_q}[\mathsf{H}_q])$, with $\mathsf{H}_i\leq \mathsf{G}\times\mathbb{Z}_2$, $\sigma_i\in\mathsf{Z}^2(\mathsf{H}_i, \mathbb{F}^*)$, $M_{k_i}(\mathbb{F}^{\sigma_i}[\mathsf{H}_i])$ is graded with an elementary-canonical $\mathsf{G}\times\mathbb{Z}_2$-grading. On the other hand, as $\mathsf{G}\cong\mathbb{Z}_p$, with $p$ odd, we have that $\mathsf{G}\times\mathbb{Z}_2$ is isomorphic to $\mathbb{Z}_p\times\mathbb{Z}_2\cong\mathbb{Z}_{2p}$. Hence, for any subgroup $\mathsf{H}$ of $\mathsf{G}\times\mathbb{Z}_2$, it follows that any $2$-cocycle $\sigma\in\mathsf{Z}^2(\mathsf{H},\mathbb{F}^*)$ is symmetric (see Remark \ref{3.25}). Consequently, the algebras $\mathbb{F}^{\sigma_s}[\mathsf{H_s}]$'s are commutative.

Now, since $\mathsf{T^G}(\mathfrak{A})=\mathsf{T^G}(\mathsf{E}^\mathsf{G}(\widehat{\mathfrak{A}}))$, it follows that $\mathsf{E}^\mathsf{G}(\widehat{\mathfrak{A}})_e$ is central in $\mathsf{E}^\mathsf{G}(\widehat{\mathfrak{A}})$, where $e$ is the neutral element of $\mathsf{G}$. Recall that $\mathsf{E}^\mathsf{G}(\widehat{\mathfrak{A}})_e=\widehat{\mathfrak{A}}_{(e,0)}\otimes_\mathbb{F}\mathsf{E}_0 +\widehat{\mathfrak{A}}_{(e,1)}\otimes_\mathbb{F}\mathsf{E}_1$. Observe that $\widehat{\mathfrak{A}}$ nilpotent implies $\mathsf{E}^\mathsf{G}(\widehat{\mathfrak{A}})$ nilpotent, and so the result follows. Suppose that $\widehat{\mathfrak{A}}$ is not nilpotent. Thus, $\mathfrak{B}_{(e,0)}\neq\{0\}$, otherwise $\widehat{\mathfrak{A}}_{(e,0)}=\mathsf{J}_{(e,0)}$, and so, since $\mathsf{J}_{(e,0)}$ is nilpotent, we must conclude that $\widehat{\mathfrak{A}}$ is nilpotent, which leads to a contradiction. It means that $\mathfrak{B}_{(e,0)}\otimes_\mathbb{F}\mathsf{E}_0\neq\{0\}$ is central in $\mathsf{E}^\mathsf{G}(\widehat{\mathfrak{A}})$. From this, it is not difficult to prove that $\mathfrak{B}_{(e,0)}$ is central in $\mathfrak{B}$, because $\mathsf{E}_0\subset\mathcal{Z}(\mathsf{E})$, and so by the item \textit{ii)} of Lemma \ref{3.04}, we deduce that  $k_1=\cdots=k_q=1$, and consequently, $\mathfrak{B}$ is equal to $\mathbb{F}^{\sigma_1}[\mathsf{H}_1] \times \cdots \times \mathbb{F}^{\sigma_q}[\mathsf{H}_q]$, which is commutative. 
\textbf{Claim:} 
$\mathsf{E}^\mathsf{G}(\mathfrak{B})=\mathfrak{B}_0\otimes_{\mathbb{F}}\mathsf{E}_{0}$ is a commutative algebra. Indeed, first, obviously $\mathfrak{B}_0\otimes_{\mathbb{F}}\mathsf{E}_{0}$ is commutative, since $\mathfrak{B}$ and $\mathsf{E}_{0}$ are commutative. Now, suppose that $\eta_{(e,1)}\in \mathbb{F}^{\sigma_s}[\mathsf{H_s}]$ for some $s=1,\dots,q$. Hence, for any $x_1,y_1\in\mathsf{E}_1$ such that $x_1 y_1\neq0$, we have that
\begin{equation*}
		\begin{split}[\eta_{(e,1)}\otimes x_1,\eta_{(e,1)}\otimes y_1]=2\sigma_s((e,1),(e,1))\eta_{(e,0)}\otimes x_1 y_1\neq0
		\end{split} \ ,
\end{equation*}
but $\eta_{(e,1)}\otimes x_1\in (\mathsf{E}^\mathsf{G}(\widehat{\mathfrak{A}}))_e$, and so $\eta_{(e,1)}\otimes x_1$ is central in $\mathsf{E}^\mathsf{G}(\widehat{\mathfrak{A}})$. This contradiction ensures that $\mathfrak{B}_{(e,1)}=\{0\}$. Analogously, suppose $\eta_{(\xi,1)}\in \mathbb{F}^{\sigma_s}[\mathsf{H_s}]$ for some $s=1,\dots,q$ and $\xi\in\mathsf{G}$ such that $(\xi,1)\in \mathsf{H}_s$. Note that $(\xi,1)^{\mathsf{o}(\xi)}=(e,1)$, because $p$ is odd, where $\mathsf{o}(\xi)$ is the order of $\xi$. Hence, for any $x_1,\dots, x_{\mathsf{o}(\xi)}\in\mathsf{E}_1$ such that $x_1\cdots x_{\mathsf{o}(\xi)}\neq0$, we have that 
\begin{equation*}
		\begin{split}
(\eta_{(\xi,1)}\otimes x_1)\cdots(\eta_{(\xi,1)}\otimes x_{\mathsf{o}(\xi)})=\lambda\eta_{(\xi,1)^{\mathsf{o}(\xi)}}\otimes x_1 \cdots x_{\mathsf{o}(\xi)}=\lambda\eta_{(e,1)}\otimes x_1 \cdots x_{\mathsf{o}(\xi)}\neq0
		\end{split} \ ,
\end{equation*}
where $\lambda=\sigma_s((\xi,1),(\xi,1))\sigma_s((\xi^2,1),(\xi,1))\cdots\sigma_s((\xi^{\mathsf{o}(\xi)-1},1),(\xi,1))$, but this contradicts the fact that $\mathfrak{B}_{(e,1)}=\{0\}$. We conclude that $\mathfrak{B}_1\otimes_{\mathbb{F}}\mathsf{E}_1=\{0\}$, and the claim is proven. Finally, using the equality $\left[\mathsf{E}^\mathsf{G}(\mathfrak{B}),\mathsf{E}^\mathsf{G}(\mathfrak{B}) \right]=\{0\}$, we have that
\begin{equation*}
		\begin{split}
\left[\mathsf{E}^\mathsf{G}(\widehat{\mathfrak{A}}),\mathsf{E}^\mathsf{G}(\widehat{\mathfrak{A}})\right] & \subseteq \left[\mathsf{E}^\mathsf{G}(\mathfrak{B}) + \mathsf{E}^\mathsf{G}(\mathsf{J}), \mathsf{E}^\mathsf{G}(\mathfrak{B}) + \mathsf{E}^\mathsf{G}(\mathsf{J}) \right] \\
		& \subseteq \left[\mathsf{E}^\mathsf{G}(\mathfrak{B}), \mathsf{E}^\mathsf{G}(\mathfrak{B}) \right] + \left[\mathsf{E}^\mathsf{G}(\mathfrak{B}) , \mathsf{E}^\mathsf{G}(\mathsf{J}) \right] + \left[\mathsf{E}^\mathsf{G}(\mathsf{J}), \mathsf{E}^\mathsf{G}(\mathfrak{B}) + \mathsf{E}^\mathsf{G}(\mathsf{J}) \right] \\
		& \subseteq \left[\mathsf{E}^\mathsf{G}(\mathfrak{B}) , \mathsf{E}^\mathsf{G}(\mathsf{J}) \right] + \left[\mathsf{E}^\mathsf{G}(\mathsf{J}), \mathsf{E}^\mathsf{G}(\widehat{\mathfrak{A}}) \right] \\
		& \subseteq \left[\mathsf{E}^\mathsf{G}(\mathsf{J}),\mathsf{E}^\mathsf{G}(\mathsf{J})\right] \subseteq \mathsf{E}^\mathsf{G}(\mathsf{J}) \ .
		\end{split}
\end{equation*}
Therefore, as $\mathsf{E}^\mathsf{G}(\mathsf{J})$ is nilpotent due to $\mathsf{J}$ being nilpotent, the result follows.
\end{proof}

%




\bibliographystyle{amsplain}
%

\end{document}